\documentclass[11pt,reqno,oneside]{amsart}
\usepackage[a4paper, total={6in, 9in}]{geometry}
\usepackage[usenames]{color}
\usepackage{amsmath,pdfsync,verbatim,graphicx,epstopdf,enumerate}
\usepackage{todonotes}
\usepackage{mathtools}  
\mathtoolsset{showonlyrefs} 
\usepackage{amsmath,amscd, amssymb, amsthm, mathrsfs}
\usepackage[abbrev,nobysame,alphabetic]{amsrefs}
\usepackage[colorlinks=true]{hyperref}
\usepackage{cancel}
\usepackage{accents}
\usetikzlibrary{patterns}

\usepackage{nicefrac}
\usepackage{comment}
\usepackage{float}
\usepackage{tikz}
	\usetikzlibrary{shapes,arrows.meta,quotes,angles,positioning,patterns,decorations.pathreplacing,calc}
\def\MarkRightAngle[size=#1](#2,#3,#4){%I hate tikz so much
 \draw ($(#3)!#1!(#2)$) -- 
       ($($(#3)!#1!(#2)$)!#1!90:(#2)$) --
       ($(#3)!#1!(#4)$);
}

\hypersetup{allcolors=blue}
 \DeclarePairedDelimiter\pair{\langle}{\rangle}
 \DeclareMathOperator{\vspan}{span}
    \newcommand{\e}{\epsilon}

\newcommand{\I}{\mathrm{i}}
\DeclareMathOperator{\tr}{\mathrm{tr}}

\newcommand{\PD}{\partial}
\newcommand{\N}{\mathbb{N}}

\newcommand{\Beq}{\begin{equation}}
	\newcommand{\Eeq}{\end{equation}}
\newcommand{\beq}{\begin{equation*}}
	\newcommand{\eeq}{\end{equation*}}
\newcommand{\bal}{\begin{align}}
	\newcommand{\eal}{\end{align}}

\newcommand{\K}{\kappa}
\usepackage{mathtools}

\newcommand{\p}{\partial}

\newtheorem{theorem}{Theorem}[section]

\newtheorem{lemma}[theorem]{Lemma}
\newtheorem{proposition}[theorem]{Proposition}

\theoremstyle{definition}
\newtheorem{definition}{Definition}[section]

\newtheorem{remark}{Remark}[section]

\newcommand{\Lc}{\mathcal{L}}
\newcommand{\R}{\mathbb{R}}

\newcommand{\Rn}{\mathbb{R}^n}
\newcommand{\norm}[1]{\lVert #1 \rVert}
\newcommand{\abs}[1]{\lvert #1 \rvert}

 \newcommand{\lr}{\langle}
\newcommand{\rn}{\rangle}

\newcommand{\supp}{\mathop{\rm supp}}

\newcommand{\df}{\mathrm{d}}
\newcommand{\nrm}[2][]{ \| {#2} \|_{#1}}
\newcommand{\agl}[1][\cdot]{ \langle {#1} \rangle}
\allowdisplaybreaks

\numberwithin{equation}{section}

\usepackage{todonotes}
\newcommand{\HOX}[1]{\todo[noline,color=white,bordercolor=white,size=\footnotesize]{#1}}

\title[]{Coefficient Determination for Non-Linear Schr\"odinger Equations on manifolds}

\author{Matti Lassas}
\address{Department of Mathematics and Statistics, University of Helsinki, Box 68, Helsinki, 00014, Finland}
\email{Matti.Lassas@helsinki.fi}
\author{Lauri Oksanen} 
\address{Department of Mathematics and Statistics, University of Helsinki, Box 68, Helsinki, 00014, Finland}
\email{lauri.oksanen@helsinki.fi }

\author{Suman Kumar Sahoo}
\address{Department of Mathematics, Indian Institute of Technology, Bombay}
\email{suman@math.iitb.ac.in}

\author{Mikko Salo}
\address{University of Jyv\"askyl\"a, Finland}
\email{mikko.j.salo@jyu.fi}
\author{ALEXANDER TETLOW}
\address{Department of Mathematics and Statistics, University of Helsinki, Box 68, Helsinki, 00014, Finland}
\email{}

\begin{document}
\begin{abstract}
We consider an inverse problem of recovering the unknown coefficients $\beta(t,x)$ and $V(t,x)$ appearing in a time-dependent nonlinear Schr\"odinger equation $ (\I \p_t +\Delta +V)u + \beta u^2=0$ in $(0,T) \times M$, on Euclidean geometry as well as on Riemannian geometry. We consider measurements in  $\Omega \subset M$ that is a neighborhood of the boundary of $M$ and the source-to-solution map $ L_{\beta, V}$ that maps a source $f$ supported in $ \Omega\times (0,T) $ to the restriction of the solution $u$ in $ \Omega\times (0,T) $. We show that the map $L_{\beta, V}$ uniquely determines the time-dependent potential and the coefficient of the non-linearity, for the above non-linear Schr\"odinger equation and for the Gross-Pitaevskii equation, with a cubic non-linear term $\beta |u|^2 \, u$, that is encountered in quantum physics.
\end{abstract}
\subjclass[2010]{Primary 35R30, 31B20, 31B30, 35J40}
	\subjclass[2020]{Primary 35R30, 31B20, 31B30, 35J40}
	\keywords{Calder\'{o}n problem, nonlinear Schr\"odinger equation, Gaussian beam}
\maketitle
\section{Introduction and main results}

Let $T>0$ and $(M,g)$ be a compact Riemannian manifold of dimension $n\ge 2$ with smooth boundary. Let $\Omega$ be a neighbourhood of boundary $\PD M$, and let $V\in C_c^{\infty}\left(0,T)\times M\setminus\Omega\right)$, and  $\beta\in C_c^{\infty}\left(0,T)\times M\setminus\Omega\right)$. Moreover, $\beta$ is non-zero almost every where in the support of $V$. We now consider the following initial boundary value problem for non-linear Schr\"odinger equation:
\begin{equation}\label{main_IBVP}
    \begin{aligned}
        \left(\I \PD_t +\Delta_g +V(t,x)\right)u(t,x)+\beta(t,x) u^2(t,x)=f(t,x) \quad &\mbox{in} \quad (0,T) \times M,\\
        u(t,x)=0 \quad & \mbox{on} \quad (0,T)\times \PD M,\\
        u(0,x)=0 \quad & \,\mbox{in} \quad\, x\in M.
    \end{aligned}
\end{equation}
Here $ \Delta_g$ stands for Laplace Beltrami operator on $M$ and in local coordinates it is given by
\begin{align*}
    \Delta_g u := |g|^{-\frac{1}{2}}\partial_{j}(|g|^{\frac{1}{2}} g^{jk} \partial_{k} u).
\end{align*}
When $g$ is the Euclidean metric, we denote $\Delta_g=\Delta= \sum_{j=1}^n \PD^2_j$.
The problem we are  interested in is to recover unknown coefficients $V$ and $\beta$ from the knowledge of source to solution map $L_{V,\beta}$ defined as:
\begin{align}\label{eq_sos_map}
    L_{V,\beta}f:= u|_{(0,T) \times \Omega},
\end{align}
where the sources are in $\{ f\in \mathcal{H}:\mbox{supp} f\subset (0,T) \times \Omega\}$. Here $\mathcal{H}$ is a sufficiently small neighbourhood of the zero function in the space $H_{00}^{2\K}$, and $u$ is
the unique solution of the non-linear Schr\"odinger equation \eqref{main_IBVP} corresponding to the source term $f$. The function space $H_{00}^{2\K}$ is defined as 
\begin{align}
    H_{00}^{2\K}:= \{f\in H^{2\K}((0,T)\times \Omega): \PD^m_tf|_{t=0}=0 \quad \mbox{for \, $0\le m \le 2\K-1$} \}.
\end{align}

The aim of this article is to prove the unique determination of the potential $V(t,x)$ and the coefficient $\beta(t,x)$ from the knowledge of source to solution map $L_{V,\beta}$ in the Euclidean setting as well as geometric setting. We now state our main result in the Euclidean setting:
\begin{theorem}\label{th:Eulidean}
    Let $n\ge 2$, $T>0$ and $M\subset \Rn$  be a convex and compact set with a  smooth boundary, and $\Omega$ be a neighbourhood of $\PD M$. Suppose $\beta_j \in C_c^{\infty}\left(0,T)\times M\setminus\Omega\right)$ and $V_j\in C_c^{\infty}\left(0,T)\times M\setminus\Omega\right)$, for $j=1,2$. Moreover assume that $\beta_j$ are non-zero everywhere in the support of $V_j$ for $j=1,2$. Then the following holds true.
    \begin{align}
        L_{\beta_1, V_1}= L_{ \beta_2,V_2} \implies (\beta_1,V_1)=(\beta_2,V_2) \quad \mbox{in $M$}.
    \end{align}
\end{theorem}
We now  state our main result on a Riemannian manifold $(M,g)$ under the following geometric assumption on $(M,g)$.
% \begin{assumption}\label{assumption_admissible}
%     We say that a point $p\in M $ is generated by admissible geodesics if there exist unit vectors $\xi_1, \xi_2\in T_pM$ with $ \agl[\xi_1,\xi_2]_g=0$ such that 
%     \[  \xi_0= \lambda'\xi_1+\lambda \xi_2,\quad \lambda'=\sqrt{1-\lambda^2},\quad \mbox{ $\lambda>0$ any small number}.\] Moreover we have $ \gamma_{p,\xi_0}\cap \gamma_{p,\xi_1}\cap \gamma_{p,\xi_2}= \{p\}.$ We further assume that each $ \gamma_{p.\xi_j}$ is non trapping, non-tangential geodesic which does not selfintersect at $p\in M$. 
% \end{assumption}
\begin{definition}\label{def_admisssible}
    We say that a manifold $M$ is admissible if every point  $p\in M$ is generated by admissible geodesics. We say that a point $p\in M $ is generated by admissible geodesics if there exist unit vectors $\xi_1, \xi_2\in T_pM$ with $ \agl[\xi_1,\xi_2]_g=0$ such that for
    \[  \xi_0= \lambda'\xi_1+\lambda \xi_2,\quad \lambda'=\sqrt{1-\lambda^2},\quad \mbox{ $\lambda>0$ any small number},\] we have $ \gamma_{p,\xi_0}\cap \gamma_{p,\xi_1}\cap \gamma_{p,\xi_2}= \{p\}.$ We further assume that each $ \gamma_{p.\xi_j}$ is non trapping, non-tangential geodesic which does not self intersect at $p\in M$. 
Here $\gamma_{p,\xi}$ is the geodesic with the initial data $(p,\xi) \in T M$.
\end{definition}
%Examples of admissible manifolds include closed, bounded domains in $\Rn$ and simple manifolds. Our second main result is as follows.
   Let us now take a moment to consider some examples of admissible manifolds. First, it is clear that if $M$ is a closed, bounded domain in $\Rn$, then it satisfies Definition \ref{def_admisssible}. Similarly, if $M$ is a simple manifold, it follows immediately that $M$ is admissible. On the other hand, if $M$ is any neighbourhood of a hemisphere in $\mathbb{S}^2$, then $M$ fails to be admissible, as every geodesic which passes through any $p\in M$ sufficiently near the boundary also passes through its antipodal point $p'$.

 However, we note that not every admissible manifold is either i) simple, or ii) a closed, bounded domain of $\Rn$. For example, given any finite interval $I$, the cylinder $M=I\times\mathbb{S}^1$ is admissible. Furthermore, so is any sufficiently small neighbourhood of the equator in $\mathbb{S}^2$ -- we leave it up to the reader to verify that if $M=\big\{(\theta,\varphi)\in\mathbb{S}^2:\ \varphi\in[-\varepsilon,\varepsilon]\big\}$, then $\varepsilon<\frac{\pi}{4}$ is sufficient for $M$ to be admissible.

Our second main result is as follows:

\begin{theorem}\label{th:main_result}
Let $(M,g)$ be an admissible manifold and $ \Omega$ be a neighbourhood of $\PD M$. Suppose $\beta_j \in C_c^{\infty}\left(0,T)\times M\setminus\Omega\right)$ and $V_j\in C_c^{\infty}\left(0,T)\times M\setminus\Omega\right)$, for $j=1,2$. Moreover assume that $\beta_j$ are non-zero everywhere in the support of $V_j$ for $j=1,2$. Then  \[  L_{\beta_1,V_1}= L_{\beta_2,V_2} \implies (\beta_1,V_1)=(\beta_2,V_2) \,\, \mbox{everywhere in $M$}.\]
\end{theorem}

The idea of proving our main results is to  use boundary sources that create special solutions of the linearised problem. In the Euclidean context, we utilize geometric optic solutions, whereas in the geometric scenario, we use Gaussian beam constructions for the linearized problem. These solutions can be combined to focus on a given point. By taking advantage of how these solutions interact with each other, we can recover the coefficients at that point. Importantly, our method does not require assuming that the geodesic ray transform is invertible on $(M,g)$, unlike the know results for the linear Schr\"odinger equation.

%The proof of the above result relies on the use of boundary sources that give rise to Gaussian beam solutions of the linearised problem. The product of these solutions can be chosen to focus at a given point, and we can then exploit the non-linear interactions of the solutions to determine the coeffients at that point. Notably, we do not need to assume that the geodesic ray transform is invertible on $M$, as we do for the linear Schr\"odinger equation.

In fact, this method can be applied, with modifications, to a variety of non-linear Schr\"odinger equations. We consider, as an example, the Gross-Pitaevskii equation in the case where $M\subseteq\R^n$ is a convex Euclidean domain:
\begin{align}\label{IBVP_GP}
\begin{cases}
i \p_t u + \Delta u + V u + \beta |u|^2u = f & \text{on $(0,T) \times M$,}
\\
u|_{x \in \p M} = 0,
\\
u|_{t = 0} = 0.
\end{cases}
    \end{align}
If we again denote the source-to-solution map for the above problem by\[L_{\beta,V}f=u|_{(0,T)\times \Omega},\]where $f$ is supported in a given neighbourhood $\Omega$ of $\p M$ and $u$ solves \eqref{IBVP_GP}, then we have the following result:
\begin{theorem}\label{th:GP}
Let $n\ge 2$ and $M\subseteq\R^n$ be a convex Euclidean domain, and let $L_{\beta,V}$ be the source-to-solution map for the Gross-Pitaevskii equation. Then  \[L_{\beta_1,V_1}=L_{\beta_2,V_2} \implies  (\beta_1,V_1)=(\beta_2,V_2) \quad \mbox{everywhere in} \quad M.\]
\end{theorem}
One can formulate the result for Gross-Pitaevskii equation on certain Riemannian manifolds like non-linear Schr\"odinger equation. However, for simplicity, we confine our focus exclusively to the Euclidean setting.
\subsection{Earlier studies and related results} The non-linear Schr\"odinger equations arise in the study of Bose-Einstein condensates \cite{Bose_Einstein_book} and the propagation of light in nonlinear optical fibers \cite{Melomad}. They also appear in the study of gravity waves on water and the models in waves in plasma \cite{Melomad}.

Literature dealing with the linearized problem of recovering the time-dependent potentials of the dynamic Schr\"odinger equation is reasonably plentiful. It was initially shown by Eskin  \cite{Eskin_JMP} that the time-dependent electromagnetic potentials are uniquely determined by the Dirichlet-to-Neumann map. Logarithmic stability estimates for this recovery were established in \cites{Ben_aicha_JMP,CKS_SIAM}   and further stability estimates of H\"older-type were established by Kian and Soccorsi \cites{Kian_Soc_SIAM,Kian_Tetlow_IPI}. Let us also mention the work of Bellassoued and O. ben Fraj \cite{Bellassoued_Ben_IPI}, which establishes logarithmic and double-logarithmic stability estimates for the same problem with partial data. In the Riemannian setting, H\"older-stable recovery of the potentials from the Dirichlet-to-Neumann was first established for time-independent potentials \cites{Bellassoued_IP,Bellassoued_Choulli_JFA,Bellassoued_Ferreira_IP}, and then for time-dependent potentials by \cite{Kian_Tetlow_IPI}. Lastly, there is the work \cite{Tetlow_SIAM}, which uniquely recovers the time- dependent Hermitian coefficients appearing in the dynamic Schr\"odinger equation on a trivial vector bundle.  
 
The inverse problem studied here is a generalization of the inverse problem introduced by Calder\'on \cite{Calderon_Paper}, where the objective is to determine the electrical conductivity of a medium by making voltage and current measurements on its boundary. It is closely related to the problem of determining an unknown potential $q(x)$ in a fixed energy Schr\"dinger operator $\Delta  + q(x)$ from boundary measurements, first solved by Sylvester and Uhlmann \cite{SYL} in dimensions $n\ge 3$ and by Bukhgeim \cite{Bukhgeim_JIIP} in the $2$- dimensional space. For the inverse conductivity problem, the first global solution in two dimensions is due to Nachman \cite{Nachman_annals} for conductivities with two derivatives and by Astala and Pa\"iva\"rinta \cite{Astala_Paivarint_Annals_Math} the uniqueness of the inverse problem was proven general isotropic conductivities in $L^{\infty}$. We refer \cite{Uhl_survey}
 for more results in this direction.

In the context of anisotropic conductivity, where the conductivity  $\gamma=(\gamma_{ij})$ is a positive definite smooth matrix, recovering $\gamma$ from the Dirichlet-to-Neumann map poses the anisotropic Calder\'on problem. In dimensions $n\geq 3$, this problem is purely geometric in nature and it is equivalent to reconstructing a Riemannian metric $g$ from $ \Lambda_g$, where $g = |\det \gamma|^{1/(n-2)} \gamma^{-1}$ (see \cite{Uhl_30_years_Cal_prob}). Another variant involves recovering a conformal factor $\alpha$ from $\Lambda_{\alpha g}$, where $\alpha$ is a smooth positive function and $g$ is fixed, akin to retrieving a Riemannian metric within the same conformal class. The analogous reduction of recovering a conformal factor to retrieving a potential $q$ from the DN map $\Lambda_q$ associated with $-\Delta_g +q=0$ is discussed in \cite{Uhl_30_years_Cal_prob}.

In dimensions $n\geq 3$, the above mentioned problem remains open, with partial solutions available for smooth manifolds exhibiting certain product structures; see \cites{DOS, Dos_Jems}.  For more general smooth Riemannian manifolds, additional results are found in \cites{uhlmann2021anisotropic,ma_sahoo_salo_anisotropic_high_fre}. While past works often rely on the injectivity or stability of the geodesic ray transform for their proofs, this article adopts a novel approach. By strategically using nonlinearity as a tool, it avoids the need for such assumptions on the geodesic ray transform, similar ideas can be found in \cites{LLLS_JMPA, LLLS_Revista,FO_jde,Krupchyk_Uhlmann_magnetic} and the references therein. %The present study follows suit, employing nonlinearity to recover both a potential term and a coefficient from the source-to-solution map.

The inverse problems for nonlinear elliptic equations have also been widely studied. A standard method is to show that the first linearization of the nonlinear Dirichlet- to-Neumann map is actually the Dirichlet-to-Neumann map of a linear equation, and to use the theory of inverse problems for linear equations. For the semilinear stationary Schro\"dinger equation $\Delta u + a(x, u) = 0$, the problem of recovering the potential $a(x, u)$ was studied in  \cites{Isakov_Sylvester_CPAM,Sun_EJDE} in dimensions $n\ge 3$, and in \cites{Imanuvilov_Yamamoto_JIIP,Isakov_Nachman_TAMS,Sun_EJDE} when
$n = 2$. In addition, inverse problems have been studied for quasilinear elliptic equations \cites{Kang_Nakamura_IP,Sun_Uhlmann_AJM,Sun_Math_Z}. Certain Calder\'on type inverse problems for quasilinear equations on Riemannian manifolds were recently considered in \cite{LLS_Math_ann}. 

This paper uses extensively the non-linear interaction of solutions to solve inverse problems. In this approach, nonlinearity is used as a tool that helps in solving inverse problems and the reconstruction applies the higher order linearizations of the source-to-solution map. An Inverse problem for a non-linear scalar wave equation with a quadratic non-linearity was studied in \cite{KLU_Invent_math} using the multiple-fold linearization and non-linear interaction of non-smooth solutions of linearized equations. For the direct problem, the analysis of non-linear interaction for hyperbolic equations started in the studies of Bony \cite{Bony-secondmicro_POS}, Melrose and Ritter \cite{Melrose_Ritter_Annals_Math}, and Rauch and Reed \cite{Rauch_Reed_CPDE}, see also \cites{Barreto_IPI,Barreto_Wang_APDE}. These studies used microlocal analysis and conormal singularities, see \cites{GrU_CMP,GuU_duke,Melrose_Uhlmann_CPAM}.
The inverse problem for a semi-linear wave equation in $(1+3)$-dimensional Lorentzian space with quadratic non-linearities was studied in \cite{KLU_Invent_math} using interaction of four waves. This approach was extended for a general semi-linear term in \cites{Hintz_Uhlmann_Zhai_IMRN,LUW_CMP} and with quadratic derivative in \cite{Wang_Zhou_CPDE}. In \cite{KLOU_duke}, the coupled Einstein and scalar field equations were studied. The result has been more recently strengthened in \cite{Uhlmann_Wang_CPAM} for the Einstein scalar field equations with general sources. The inverse for semi-linear and quasi-linear wave equations in $(1 + n)$-dimensional space are studied in \cite{FLO_Forum_math} using the three wave interactions. In \cites{FO_jde,LLLS_JMPA,LLLS_Revista} similar multiple-fold linearization methods have been introduced to study inverse problems for elliptic non-linear equations, see also \cites{KrU_MRL,KrU_PMAS,SKSG_biharmonic_nonlinear}.

In recent works \cites{CLOP_JEMS,Chen_Lassas_Oksanen_Paternain_Yang_Mills,FO_JIMJ}, the authors have also studied problems of recovering zeroth and first order terms for semi-linear wave equations with Minkowski metric. The three wave interactions were used in \cites{CLOP_JEMS,Chen_Lassas_Oksanen_Paternain_Yang_Mills} to determine the lower order terms in the equations and in modelling non-linear elastic scattering from discontinuities \cites{deHoop_Uhlmann_Wang_IHP,deHoop_Uhlmann_Wang_Math_ann}.

For the linear wave equation, the determination of general time-dependent coefficients have been studied using the propagation of singularities. In the studies of recovery of sub-principal coefficients for the linear wave equation, we refer the reader to the recent works \cites{FIKO_spectral_theory,FIO_JGA,Stefanov_PAMS} for recovery of zeroth and first order coefficients and to \cite{Stefanov_Yang_APDE} for a reduction from the boundary data for the inverse problem associated to linear wave equation to the study of geometrical transforms of the domain. This latter approach has been recently extended to general real principal type differential operators \cite{real_principal_inv_prob}. Let us also mention here the recent works \cites{AAL_arxiv,AAL_Invent_math} which recover zeroth order coefficients of the wave equation on Lorentzian manifolds from the Dirichlet- to-Neumann map under suitable geometric assumptions.

The remainder of the article is structured as follows. In Section \ref{e_s2}, we establish the well-defined and smooth nature of the source-to-solution map  \eqref{eq_sos_map} in the vicinity of the zero solution. Moving on to Section \ref{e_s3}, we initially provide the proof of Theorem \ref{th:Eulidean} in the context of Euclidean geometry, aiming to enhance accessibility to the argument. In Section \ref{e_s4}, we demonstrate how this proof can be adapted for the Gross-Pitaevskii equation. Finally, we present the proof of Theorem \ref{th:main_result} in the geometric case in Section \ref{e_s5}.

%The rest  of the article  is organised as follows. In Section \ref{e_s2} we show that the source-to-solution map for \eqref{eq_sos_map} is well-defined and smooth in a neighbourhood of zero solution. In Section \ref{e_s3}, we first give the proof of Theorem \ref{th:Eulidean} in the case of Euclidean geometry, in the hope that this will make the argument more accessible. In Section \ref{e_s4}, we show how this proof can be modified for the Gross-Pitaevskii equation. Finally in Section \ref{e_s5}, we present the proof of  Theorem \ref{th:main_result} in the geometric case.

%\section{The source-to-solution map}

\section{The Source-to-Solution Map}\label{e_s2}
The aim of this section is to establish that the source-to-solution map for the problem \eqref{main_IBVP}  is well defined  and smooth in a neighbourhood of zero in the Euclidean setting. In geometric setting similar argument works. %More precisely, let $\kappa\in\mathbb{N}$ be sufficiently large, and let $\mathcal{H}$ be a small enough neighbourhood of zero in the space\[H^{2\kappa}_0:=\{f\in H^{2\kappa}((0,T)\times M): \p_t^mf|_{t=0}=0\ \forall\, m\leq 2\kappa-1\}.\]
%Then we have the result below:
More precisely we prove the following:
\begin{proposition}\label{e_prop}
Fix $\beta,V$ in the equation \eqref{main_IBVP}. Then for any $f\in\mathcal{H}$ there is a unique $u\in H^{2\kappa}_{00}$ such that $u=L_{\beta,V}f$, and the map $L_{\beta,V}:\mathcal{H}\rightarrow H^{2\kappa}_{00}$ is smooth.
\end{proposition}

In order to prove this result, we need some higher order energy estimates for the linearized problem. Thus, we begin by recalling the inhomogeneous linear Schr\"odinger equation
\begin{align}\label{e_LS}
\begin{cases}
(i\p_t+\Delta+V) u = f & \text{on $(0,T) \times M$,}
\\
u|_{x \in \p M} = 0,
\\
u|_{t = 0} = 0,
\end{cases}
    \end{align}
and let $\mathcal{S}$ denote the solution operator for the above equation, defined by $\mathcal{S}(f)=u$.
Let us also define the energy space\[H^{r,s}((0,T)\times M)=H^r(0,T;L^2(M))\cap L^2(0,T;H^s(M)),\] together with the associated norm $\norm{\cdot}_{H^{r,s}((0,T)\times M)}=\norm{\cdot}_{H^r(0,T;L^2(M))}+\norm{\cdot}_{L^2(0,T;H^s(M))}$.

We recall here the usual energy estimates for the linearized problem (\ref{e_LS}), which hold under the assumption that $f|_{t=0}=0$  (for details, see for example \cite{Kian_Soc_SIAM}). These estimates are:
\begin{align}\label{e_NRG1}
&\norm{u}_{L^\infty(0,T;L^2(M))}\leq C\norm{f}_{L^2((0,T)\times M)},
\\
\label{e_NRG2}
&\norm{\p_tu}_{L^\infty(0,T;L^2(M))}\leq C\norm{f}_{H^{1,0}((0,T)\times M)}
, \,\, \norm{\Delta u}_{L^2((0,T)\times M)}\leq C\norm{f}_{H^{1,0}((0,T)\times M)}.
%\label{e_NRG3}
\end{align}
We now establish the desired higher order energy estimates for the linearised problem (\ref{e_LS}).
\begin{lemma}\label{e_NRG_Higher}
The problem \eqref{e_LS} satisfies the estimate\begin{equation}\label{e_NRG4}\norm{u}_{H^{2\kappa}((0,T)\times M)}\leq C\norm{f}_{H^{2\kappa}((0,T)\times M)}\end{equation}for any choice of source term $f\in H^{2\kappa}_{00}$.
\end{lemma}
\begin{proof}
We begin by noting that $i\p_tu=f-\Delta u - Vu$. Then the assumption that $\p_t^mf|_{t=0}=0$ for $m \leq 2\kappa-1$, together with the fact that $u|_{t=0}=0$, immediately implies that\begin{equation}\label{e_IV}\p_t^mu|_{t=0}=0,\ {\rm when}\ m\leq 2\kappa.\end{equation}

The proof of the estimate (\ref{e_NRG4}) is by induction. The case $\kappa=0$ is implied by (\ref{e_NRG1}). Therefore, suppose that we have shown the estimate (\ref{e_NRG4}) holds for $\kappa\leq K-1$. Then it suffices to show that, for $\rho,\sigma\in\N$ such that $\rho+2\sigma=2K$, we have\begin{equation}\label{aim}\norm{{\p_t}^\rho\Delta^\sigma u}_{L^2((0,T)\times M)}\leq\norm{f}_{H^{2K}((0,T)\times M)}.\end{equation}
We begin by applying $\p_t$ to (\ref{e_LS}), and observe that\begin{equation}\label{d_t}(i\p_t+\Delta+V)\p_tu=\p_tf-(\p_tV)u.\end{equation}Then, since $\p_tu$ satisfies the zero initial condition $\p_tu|_{t=0}=0$, we can apply (\ref{e_NRG2}) to equation (\ref{d_t}) to observe that\[\norm{\p_t^2u}_{L^\infty(0,T;L^2(M))}\leq C\norm{f}_{H^{2,0}((0,T)\times M)}+C\norm{(\p_tV)u}_{H^{1,0}((0,T)\times M)},\]and using the fact that $V\in C^\infty_0((0,T)\times M\setminus\Omega)$, together with the estimates \eqref{e_NRG1} and \eqref{e_NRG2}, we conclude that\[\norm{\p_t^2u}_{L^\infty(0,T;L^2(M))}\leq C\norm{f}_{H^{2,0}((0,T)\times M)}.\]
We can then proceed to apply this estimate to \eqref{d_t} and use the previously derived estimates to obtain the bound\[\norm{\p_t^3u}_{L^\infty(0,T;L^2(M))}\leq C\norm{f}_{H^{3,0}((0,T)\times M)},\]and bootstrapping in this manner we can conclude that\begin{equation}\label{e_b0}\norm{\p_t^mu}_{L^\infty(0,T;L^2(M))}\leq C\norm{f}_{H^{m}((0,T)\times M)}.\end{equation}
In particular, we note that estimate \eqref{e_b0} holds for all $m\leq 2K$, not just when $m$ is even, and this estimate establishes \eqref{aim} in the case where $\sigma=0$. Let us now consider the case where $\sigma=1$. Since $u$ satisfies the Schr\"odinger equation, it follows that\begin{equation}\begin{split}\label{e_b1}&\norm{\p_t^{2K-2}\Delta u}_{L^2((0,T)\times M)}\leq C\norm{\p_t^{2K-2}(f-i\p_t u+Vu)}_{L^2((0,T)\times M)}\\\leq& C\norm{f}_{H^{2K-2}((0,T)\times M)}+C\norm{\p_t^{2K-1}u}_{L^2((0,T)\times M)}+C\norm{\p_t^{2K-2}(Vu)}_{L^2((0,T)\times M)}.\end{split}\end{equation}

From \eqref{e_b0}, we deduce that the second term on the right-hand side of \eqref{e_b1} is bounded by $\norm{f}_{H^{2K-1}((0,T)\times M)}$. Further, since $V\in C^\infty_0((0,T)\times M\setminus\Omega)$, the induction hypothesis implies that the third term is bounded by $\norm{f}_{H^{2K-2}((0,T)\times M)}$. Therefore, it follows that \[\norm{\p_t^{K-2}\Delta u}_{L^2((0,T)\times M)}\leq C\norm{f}_{H^{2K-1}((0,T)\times M)},\] which establishes (\ref{aim}) in the case where $\sigma=1$. It remains only to deal with the case where $\sigma\geq2$. In this case, note that\[\begin{split}{\p_t}^\rho\Delta^\sigma u&=\, {\p_t}^\rho \Delta^{\sigma-1}(f-i\p_tu-Vu)\\&=\, {\p_t}^\rho\Delta^{\sigma-1}(f-Vu)-i{\p_t}^{\rho+1}\Delta^{\sigma-2}\Delta u\\&=\, {\p_t}^\rho\Delta^{\sigma-1}(f-Vu)-i{\p_t}^{\rho+1}\Delta^{\sigma-2}(f-i\p_tu-Vu).\end{split}\]
 Since the derivatives of $u$ and $f$ in the last expression are all of order $2K-2$ or lower, the induction hypothesis then implies that\[\norm{{\p_t}^\rho\Delta^\sigma u}_{L^2((0,T)\times M)}\leq C\norm{f}_{H^{2K-2}((0,T)\times M)},\]and this finishes the proof of (\ref{aim}) for $\sigma\geq2$.
\end{proof}
In light of the above, we can now proceed to the proof of proposition \ref{e_prop}.
\begin{proof}[Proof of Proposition \ref{e_prop}]
Let us first address the issue of uniqueness. Suppose that for some $f$, we have two solutions of \eqref{main_IBVP}, which we denote by $u$ and $v$. Then it follows that $(i\p_t+\Delta+V)(u-v)+\beta(u+v)(u-v)=0$, and applying the energy estimate \eqref{e_NRG1} for the linear problem, we conclude that $u-v=0$. This implies $L_{\beta,V}$ is well-defined.

We now fix some $\kappa\in\N$ large enough that $H^{2\kappa}$ is a Banach algebra. It follows from the trace theorem that $H^{2\kappa}_{00}$ is likewise a Banach algebra, and we can therefore define the map\[\mathcal{K}:H^{2\kappa}_{00}\times H^{2\kappa}_{00}\rightarrow H^{2\kappa}_{00}\quad \mbox{via the expression} \quad \mathcal{K}(u,f)=f-\beta u^2.\]
We now consider the map $\Phi(u,f)=u-\mathcal{S}\mathcal{K}(u,f)$, and observe that $\Phi(u,f)=0$ implies that $u$ is a solution of the non-linear Schr\"odinger equation \eqref{main_IBVP}. Observe also that
\[\mathcal{S}:H^{2\kappa}_{00}\rightarrow H^{2\kappa}_{00}\]
by the result of Lemma \eqref{e_NRG_Higher} together with \eqref{e_IV}. Therefore it follows that
\[\Phi(u,f):H^{2\kappa}_{00}\times H^{2\kappa}_{00}\rightarrow H^{2\kappa}_{00}.\]

We note that the map $\Phi$ is smooth in $u$ and $f$, since $\mathcal{K}(u,f)$ is a polynomial and $\mathcal{S}$ is linear. We can use the chain rule to compute $\p_u\Phi(0,0)=\mathrm{Id}$, and the implicit function theorem gives a smooth map $f\mapsto u$ from a neighbourhood $\mathcal{H}$ of the zero function in $H^{2\kappa}_{00}$ to $H^{2\kappa}_{00}$, such that we have $\Phi(u(f), f)=0$ for all $f\in\mathcal{H}$. This map must coincide with $L_{\beta,V}$ in $\mathcal{H}$, by the uniqueness already established.
\end{proof}

\section{The Euclidean Case}\label{e_s3}
We fix $T>0$ and consider the case where $M$ is a convex Euclidean domain of $\R^n$. We shall write $(t,x)$ for the usual Cartesian coordinates of $(0,T)\times M$.  Let us first recall the non-linear Schr\"odinger equation \eqref{main_IBVP}:
\begin{align}\label{NLS}
\begin{cases}
i \p_t u + \Delta u + V u + \beta u^2 = f & \text{on $(0,T) \times M$,}
\\
u|_{x \in \p M} = 0,
\\
u|_{t = 0} = 0.
\end{cases}
    \end{align}
    
Before entering into the technical details of the proof of Theorem \ref{th:Eulidean}, we will first give a brief explanation of how the non-linearity is used. Let $f_1,f_2\in H^{2k}_0$ be supported in $(0,T)\times\Omega$, and consider the two-parameter family of source terms
\begin{equation}f_\varepsilon:=\varepsilon_1f_1+\varepsilon_2f_2,\quad\forall(\varepsilon_1,\varepsilon_2)\in\R^2.
\end{equation} 
For small enough $\varepsilon_1,\varepsilon_2$ it follows that $f_\varepsilon\in\mathcal{H}$, and we let $w_\varepsilon$ denote the unique solution of \eqref{NLS} with this choice of source term. Then\[w:=\p_{\varepsilon_1}\p_{\varepsilon_2}w_\varepsilon|_{\varepsilon=0},\]
solves the linear Schr\"odinger equation \eqref{e_LS} with $f=-2\beta\mathcal{U}_1\mathcal{U}_2, \mbox{where }\, \mathcal{U}_j=\p_{\varepsilon_j}w_\varepsilon|_{\varepsilon=0}$, and $\mathcal{U}_j$ satisfies the \eqref{e_LS} with $f=f_j$. In the proof of Theorem \ref{th:Eulidean}, we choose $f_j$ which generate geometric optics solutions $\mathcal{U}_j$ supported near lines which intersect at some $p\in M$. As a result of this, we hope to recover information about the coefficients $(\beta,V)$ at the point $p\in M$.

The remainder of this section is divided as follows. We briefly recall the construction of approximate geometric optics solutions for the linear Schr\"odinger equation in section \ref{e_s3.1}. In section \ref{e_s3.2}, we use the convexity of $M$ to show that the source-to-solution map determines the amplitudes of these solutions in $\Omega$, and hence also the sources $f_j$ up to a small error. Finally, in section \ref{e_s3.3}, we show how the solutions generated by these sources can be used to recover the coefficients.

%%%%%%%%%%%%%%%%%%%%%%%%%%%%%%%%%%%%%%%%%%%%%%%%%%%%%%%%%%%%%%%%%%%%%%%%%%

\subsection{Geometric Optics Solutions}\label{e_s3.1}
In this section, we recall the construction of  geometric optics solutions to the linear Schr\"odinger equation. The details are largely the same as those elsewhere in the literature (e.g. \cite{Kian_Soc_SIAM}), but we give them below for the reader's convenience. Let us begin by considering the homogeneous linear Schr\"odinger equation
\begin{equation}\begin{split}\label{e_HLS}i \p_t u + \Delta u + V u &= 0 \quad\:\mbox{in}\quad (0,T)\times M\\
u|_{t=0}& = 0\quad \mbox{ in }\quad M.\end{split}\end{equation}
The construction is based on the use of the ansatz\[U(t,x)=e^{i\tau(\xi\cdot x-c\tau t)}a(\tau;t,x)=e^{i\tau(\xi\cdot x-c\tau t)}\Bigg(\sum_{k=0}^N\frac{a_k(t,x)}{\tau^k}\Bigg),\]
where $\tau>0$ is a large parameter, $c>0$ is some constant, $\xi\in\R^{n}$, and the amplitudes  $a_k$ are   to be determined. We insert $U$  to the Schr\"odinger operator to deduce\begin{equation}\label{Su}\begin{split}(i\p_t+\Delta+V)U=e^{i\tau(\xi\cdot x-c\tau t)}(i\p_t+\Delta+V)a+\tau^2e^{i\tau(\xi\cdot x-c\tau t)}(c-|\xi|^2)a\\+2i\tau e^{i\tau(\xi\cdot x-c\tau t)}(\mathcal{T}_\xi a),\end{split}\end{equation}where $\mathcal{T}_\xi=\sum_{l=1}^n\xi^l\p_{x^l}$ is the transport operator in the $\xi$ direction. We require the right-hand side of the above to vanish in powers of $\tau$. In particular, this imposes the condition that\begin{equation}\label{eikonal}|\xi|^2=c,\end{equation}and that the amplitude functions $a_k$ satisfy the transport equations
\begin{equation}\label{transports}\mathcal{T}_\xi a_0=0, \quad %\mathcal{T}_\xi a_1&=\frac{i}{2}(i\p_t+\Delta+V)a_0 \\ \vdots \\ 
\mathcal{T}_\xi a_N=\frac{i}{2}(i\p_t+\Delta+V)a_{N-1}\quad \mbox{for $N=1,2,\cdots$}.
\end{equation}
Let us now fix some $y\in M$, and denote by $\gamma_{y,\xi}$ the line through the point $y$ with direction $\xi$, as parameterized by $\gamma_{y,\xi}(s)=s\xi+y$. We can choose vectors $\omega_l\in\R^n$ such that $\{ \frac{\xi}{|\xi|},\omega_1,\cdots,\omega_{n-1}\}$ %$ \[\frac{\xi}{|\xi|},\omega_1,\cdots,\omega_{n-1}\]
forms an orthonormal basis of $\R^n$ with respect to the Euclidean metric. Then, for some small $\delta>0$, we choose the zeroth amplitude to be\begin{equation}\label{amp0}a_0(t,x)=\phi(t)\prod_{l=1}^{n-1}\chi_\delta(\omega_l\cdot(x-y)),\end{equation}where $\chi_\delta\in C^\infty_0(-\delta,\delta)$, and $\phi\in C^\infty_0((0,T))$ is a smooth cutoff. Therefore it follows that for all $t\in(0,T)$ the amplitude $a_0(t,\cdot)$ is supported in a $\delta$-neighbourhood of the line $\gamma_{y,\xi}(\R)$.

We can then use the transport equations \eqref{transports}, with vanishing initial conditions imposed upon the subset $\Sigma_{y,\xi}=\{x\in \R^n: \xi\cdot(x-y)=0\}$ to compute the remaining amplitudes $a_k$ for $k\geq1$ iteratively. Thus we have that\begin{equation}\label{ampsk}a_k(s\xi+y)=\frac{i}{2}\int_0^s\big[(i\p_t+\Delta+V)a_{k-1}\big](\tilde{s}\xi+y)d\tilde{s}.\end{equation}
It follows from the above that $U(t,\cdot)$ is compactly supported in a $\delta$-neighbourhood of $\gamma_{y,\xi}(\R)$ for all $t\in(0,T)$. Then, by using \eqref{eikonal} and \eqref{transports} in the expression \eqref{Su}, we can deduce that\[(i\p_t+\Delta+V)U=\tau^{-N}e^{i\tau(\xi\cdot x-c\tau t)}(i\p_t+\Delta+V)a_N,\]and then a direct computation shows that\[\norm{(i\p_t+\Delta+V)U}_{H^\sigma((0,T)\times M)}\lesssim\tau^{-N+2\sigma}.\]
We now convert the approximate solution $U(t,x)$  into a exact solution of   \eqref{e_HLS} by writing $ u=U+R_\tau,$ %\begin{equation}\label{Xact}u=U+R_\tau,\end{equation}
where the remainder term $R_\tau$ solves
\begin{align}\begin{cases}\label{remainder}
&(i\p_t+\Delta+V)R_\tau=-(i\p_t+\Delta+V)U\quad \text{in $(0,T) \times M$,}
\\
&R_\tau|_{x \in \p M} = 0,
\\
&R_\tau|_{t = 0} = 0.
    \end{cases}\end{align}
Then the energy estimate \eqref{e_NRG4} immediately implies that
\begin{equation}\label{e_NRG_rem}
\nrm[H^\sigma((0,T)\times M)]{R_{\tau}}=\norm{u-U}_{H^\sigma((0,T)\times M)}\lesssim\tau^{-N+2\sigma},
\end{equation}
for even $\sigma\in\N$, and we have verified that the ansatz $U(t,x)$ is indeed an approximate solution of the linear Schr\"odinger equation.

%%%%%%%%%%%%%%%%%%%%%%%%%%%%%%%%%%%%%%%%%%%%%%%%%%%%%%%%%%%%%%%%%%%%%%%%%%%

\subsection{Determination of the Boundary Sources}\label{e_s3.2}
We now demonstrate that the amplitudes of geometric optics solutions are determined in $(0,T)\times\Omega$ by the source-to-solution map $L_{\beta,V}$. We begin by defining the map $\mathcal{L}_V$ as \[\mathcal{L}_Vf=u|_{(0,T)\times\Omega},\]where $f$ is supported in $(0,T)\times\Omega$, and $u$ solves the linear Schr\"odinger equation
\begin{align}
\begin{cases}
(i\p_t+\Delta+V) u = f & \text{on $(0,T) \times M$,}
\\
u|_{x \in \p M} = 0,
\\
u|_{t = 0} = 0.
\end{cases}
    \end{align}
Note that $L_{\beta,V}$ determines $\mathcal{L}_V$ for any $\beta,V$ via the expression $\mathcal{L}_V  f=\p_\e L_{\beta,V}(\e f)|_{\e=0}.$ This shows that 
\begin{align}
    L_{\beta_1,V_1}= L_{\beta_2,V_2}\implies \mathcal{L}_{V_1}=\mathcal{L}_{V_2}.
\end{align}
We next consider the Sch\"odinger equation which is backward in time.
\begin{align}
\begin{cases}
(i\p_t+\Delta+V) w = h & \text{on $(0,T) \times M$,}
\\
w|_{x \in \p M} = 0,
\\
w|_{t = T} = 0.
\end{cases}
    \end{align}
Then, it holds that $\mathcal{L}_V^\ast h=w|_{(0,T)\times\Omega}$ with $h$ supported in $(0,T)\times\Omega$. In fact, we can compute that for source terms $f,h$ supported in $(0,T)\times\Omega$ there holds
\begin{equation}\label{adjoint}\pair{\mathcal{L}_{V}f,h}_{L^2((0,T)\times\Omega)}=\pair{f,w}_{L^2((0,T)\times\Omega)}.
\end{equation}

For $j=1,2$, let us consider the potentials $V_j\in C^\infty_0((0,T)\times M\setminus\Omega)$. Given any line $\gamma_{q,\xi}$ with initial point $q\in\p M$ and initial direction $\xi\in\R^n$, we can define a sequence of functions $a_k^{(j)}$ corresponding to $V_j$ as follows. First, let us choose vectors $\omega_1,\cdots,\omega_{n-1}\in\R^n$ such that $\{\frac{\xi}{|\xi|},\omega_1,\cdots,\omega_{n-1}\}$ 
is an orthonormal basis of $\R^n$ with respect to the Euclidean metric. Then, for some small $\delta>0$, we once again choose the zeroth amplitude
\begin{align}\label{eq_a0}
  a_0^{(j)}(t,x)=\phi(t)\prod_{l=1}^{n-1}\chi_\delta(\omega_l\cdot(x-q)),  
\end{align}
where $\chi_\delta\in C^\infty_0(-\delta,\delta)$, and $\phi\in C^\infty_0(0,T)$ is a smooth cutoff. We can define the subsequent functions by solving the following transport equations\begin{equation}\label{sequence}2i\mathcal{T}_\xi a_{k+1}^{(j)}+(i\p_t+\Delta+V_j)a_{k}^{(j)}=0,\quad a_{k+1}^{(j)}|_{\Sigma_{q,\xi}}=0.\end{equation}

We can define, for each $\tau>0$ and $N\in\N$, an approximate geometric optics solution $U_j$ corresponding to the choice of $a_0^{(j)}$ and $\gamma_{q,\xi}$ through the expression\[U_j(t,x)=e^{i(\tau\xi\cdot x-|\xi|^2\tau^2t)}\sum_{k=0}^N\tau^{-k}a_k^{(j)}(t,x).\]We denote by $u_j$ the corresponding exact solution of the Schr\"odinger equation
\begin{align}
\begin{cases}
(i\p_t+\Delta+V_j) u_j = 0 & \text{on $(0,T) \times M$,}
\\
u_j|_{t = 0} = 0,
\end{cases}
\end{align}
as constructed in the  subsection \ref{e_s3.1}. Note that $U_j$ coincides with $u_j$ up to a small error $O(\tau^{-N})$ in $L^2$. Let $\eta\in C^\infty_0(M)$ satisfy $\eta=1$ in $M\setminus\Omega$, and choose $f_j=(i\p_t+\Delta+V_j)(\eta u_j)$. Then, we note that the function $\eta u_j$ solves the Schr\"odinger equation
\begin{align}
\begin{cases}
(i\p_t+\Delta+V_j) \mathcal{U}= f_j & \text{on $(0,T) \times M$,}
\\
\mathcal{U}|_{x \in \p M} = 0,
\\
\mathcal{U}|_{t = 0} = 0,
\end{cases}
\end{align}
with the source $f_j$ supported in $(0,T)\times\Omega$. In a similar manner, we can consider the geometric optics solution of the backwards-in-time problem, 
\begin{align}
\begin{cases}
(i\p_t+\Delta+V_2) w = 0 & \text{on $(0,T) \times M$,}
\\
w|_{t = T} = 0,
\end{cases}
\end{align}
and observe that $w$ can be similarly approximated to $O(\tau^{-N})$ by the expression \[e^{i(\tau\xi\cdot x-|\xi|^2\tau^2t)}\sum_{k=0}^N\tau^{-k}w_k(t,x),\]where $w_k$ is a sequence which satisfies \eqref{sequence} with $V=V_2$ for  $k\ge 0$ and $w_0$ is given by \eqref{eq_a0}. Further, letting $\tilde{\eta}\in C^\infty_0(M)$ also satisfy $\eta=1$ in $M\setminus\Omega$, we observe that, for $h=(i\p_t+\Delta+V_2)(\tilde{\eta}w)$, the function $\tilde{\eta}w$ solves
\begin{align}
\begin{cases}
(i\p_t+\Delta+V_2) \mathcal{W} = h & \text{on $(0,T) \times M$,}
\\
\mathcal{W}|_{x \in \p M} = 0,
\\
\mathcal{W}|_{t = T} = 0.
\end{cases}
\end{align}

\begin{lemma}\label{e_lemma}
Suppose that $\mathcal{L}_{V_1}=\mathcal{L}_{V_2}$. Then $a_k^{(1)}=a_k^{(2)}$ in $(0,T)\times\Omega$ for all $k\in\N$. 
\end{lemma}

\begin{proof}
 The result is trivial  for $k=0$. We next argue by induction and  assume that the result holds for all $k\leq K-1\ll N$. We note that $a_K^{(1)}=a_K^{(2)}$ along $\gamma_{q,\xi}$ until the line leaves $\Omega$, since $V_1=V_2=0$ in this region. For the inductive step, let us first observe that $[(i\p_t+\Delta+V_j),\eta]=2\nabla\eta\cdot\nabla+\Delta\eta$. This further entails  \begin{equation}(i\p_t+\Delta+V_2)\label{source_w}\Big(\tilde{\eta}e^{i(\tau\xi\cdot x-|\xi|^2\tau^2t)}\sum_{k=0}^N\tau^{-k}w_k\Big)=e^{i(\tau\xi\cdot x-|\xi|^2\tau^2t)}\big(2i\tau(\mathcal{T}_\xi\tilde{\eta})w_0+O(1),\big)\end{equation}and further, modulo a small error of $O(\tau^{-N})$, that
 \begin{equation}
 \begin{split}\label{source_a}&(i\p_t+\Delta+V_j)\Big(\eta e^{i(\tau\xi\cdot x-|\xi|^2\tau^2t)}\sum_{k=0}^N\tau^{-k}a_k^{(j)}\Big)\\&\quad=e^{i(\tau\xi\cdot x-|\xi|^2\tau^2t)}\big(2i\tau(\mathcal{T}_\xi\eta)+2\nabla\eta\cdot\nabla+\Delta\eta\big)\sum_{k=0}^N\tau^{-k}a_k^{(j)}.
 \end{split}
 \end{equation}
Then, using the fact that $\mathcal{L}_{V_1}=\mathcal{L}_{V_2}$, the identity \eqref{adjoint} implies that
\begin{equation}\label{not_Alessandrini}0=\pair{(\mathcal{L}_{V_j}-\mathcal{L}_{V_2})f_j,h}_{L^2((0,T)\times\Omega)}=\pair{\eta u_j,h}_{L^2((0,T)\times\Omega)}-\pair{f_j,\tilde{\eta} w}_{L^2((0,T)\times\Omega)}.
\end{equation}
Taking the difference of the expression \eqref{not_Alessandrini} for $j=1$ and $j=2$, we deduce that
\begin{align}
&0=\big\langle\eta\tau^{-K}(a_K^{(1)}-a_K^{(2)})+O(\tau^{-K-1}),2i\tau(\mathcal{T}_\xi\tilde{\eta})w_0+O(1)\big\rangle\\&-\big\langle(2i\tau(\mathcal{T}_\xi\eta)+2\nabla\eta\cdot\nabla+\Delta\eta)\big(\tau^{-K}(a_K^{(1)}-a_K^{(2)})+O(\tau^{-K-1})\big),\tilde{\eta}w_0+O(\tau^{-1})\big\rangle.
\end{align}
%\[\begin{split}.\end{split}\]
By considering the leading order term, we can deduce that\[0=\big\langle(\eta\mathcal{T}_\xi\tilde{\eta}+\tilde{\eta}\mathcal{T}_\xi\eta)(a_K^{(1)}-a_K^{(2)}),w_0\big\rangle_{L^2((0,T)\times\Omega)},\]and if we choose $\eta$ so that $\eta=1$ in the support of $\tilde{\eta}$, we then have
\begin{equation}\label{final}\pair{\mathcal{T}_\xi\tilde{\eta}(a_K^{(1)}-a_K^{(2)}),w_0}_{L^2((0,T)\times\Omega)}=0.\end{equation}

We then fix a set $Q$ such that $\overline{M\setminus\Omega}\subseteq Q\subseteq M^{int}$, with smooth boundary $\p Q$ such that the line $\gamma_{q,\xi}$ intersects $\p Q$ exactly twice, and such that, near the point $x$ where $\gamma_{q,\xi}$ exits $Q$, the hyperplane $\Sigma_{x,\xi}$ coincides with $\p Q$. It follows from the convexity of $M$ that we can always construct a set with these properties as demonstrated in the figure below.

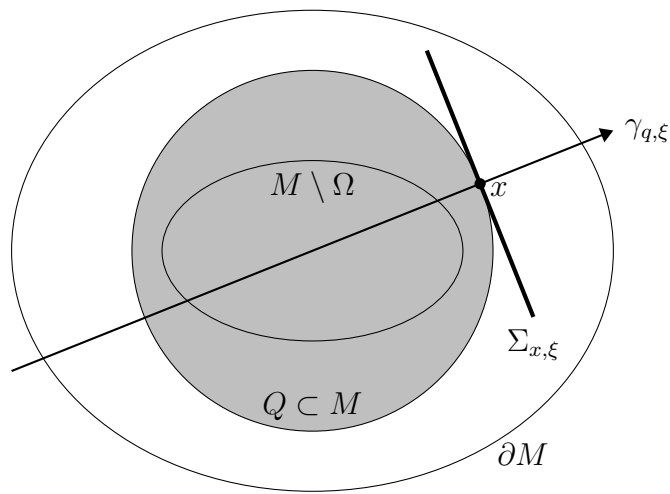
\begin{figure}[H]
\begin{center}
\begin{tikzpicture}[scale=0.8]
\draw [fill=gray!50](0,0) circle(3) node[above,yshift=-2.4cm]{$Q\subset M$};
\draw (0,0) circle (2.5 and 1.5) node[below,yshift=1.2cm]{$M\setminus\Omega$};
\draw (0,0) circle (5 and 4) node[xshift=2.8cm,yshift=-2.7cm]{$\p M$};
\draw [thick,->,>={Triangle}](-5,-2) -- (5,2) node[right]{$\gamma_{q,\xi}$};
\coordinate [circle,fill,inner sep=1.5pt,label={[xshift=2.5mm,yshift=-3.9mm]$x$}] (M) at ($(0,0)!3cm!(5,2)$);
\coordinate (A) at ($(M)!1!90:(5,2)$);
\coordinate (B) at ($(M)!1!270:(5,2)$);
\draw [ultra thick](A) -- (B) node[below]{$\Sigma_{x,\xi}$};
\end{tikzpicture}
\caption{The set $Q$, shaded in gray.}\label{Fig1}
\end{center}
\end{figure}

We can then choose $\tilde{\eta}$ which converges to the indicator function of $Q$, so that, in a neighbourhood of $x$, $\mathcal{T}_\xi\tilde{\eta}$ converges to the Dirac delta distribution on $\Sigma_{x,\xi}$. Observe that we can then choose $w_0=\phi_1\phi_2$, where $\phi_1$ converges to the delta distribution at $t\in(0,T)$ and $\phi_2$ converges to the delta distribution on $\gamma_{q,\xi}$. Thus, we can conclude from \eqref{final} that $a_K^{(1)}=a_K^{(2)}$, as required. This completes the induction argument and the proof of Lemma \ref{e_lemma}.
\end{proof}

%%%%%%%%%%%%%%%%%%%%%%%%%%%%%%%%%%%%%%%%%%%%%%%%%%%%%%%%%%%%%%%%%%%%%%%%%%%

\subsection{Proof of Theorem \ref{th:Eulidean}}\label{e_s3.3}
%We consider once again the problem \eqref{main_IBVP} for $V\in C^\infty_0((0,T)\times M\setminus\Omega)$ and $\beta\in C^\infty_0((0,T)\times M\setminus\Omega)$ which is non-zero almost everywhere in $\supp(V)$.

We begin by fixing some $p\in M\setminus\Omega$ as well as some small $\lambda>0$. We choose vectors $\xi_0$, $\xi_1$, and $\xi_2$, with the magnitude of $\xi_1$ and $\xi_2$ depending on $\lambda$, and the magnitude and direction of $\xi_0$ depending on $\lambda$. In particular, we choose such vectors so that they satisfy\[\xi_0=\xi_1+\xi_2,\]with $\xi_0$, $\xi_1$ and $\xi_2$ pairwise non-colinear, and such that we have\[|\xi_0|^2=1,\quad|\xi_1|^2=1-\lambda^2,\quad|\xi_2|^2=\lambda^2.\]
\begin{figure}[ht]
\begin{center}
\begin{tikzpicture}[scale=0.9]
\coordinate [circle,fill,inner sep=1.5pt,label={[label distance=0.1mm]0:$p\in M$}](A) at (0,0);
\coordinate (B) at (0,4);
\coordinate (C) at (-2,4);
\draw [->,>={Triangle[scale width=0.6]}](A) -- (C) node[midway,yshift=-1.5mm,xshift=-1.5mm]{$\xi_0$};
\draw [->,>={Triangle[scale width=0.6]}](A) -- (B) node[midway,xshift=2.6mm]{$\xi_1$};
\draw [->,>={Triangle[scale width=0.6]}](B) -- (C) node[midway,yshift=-2.5mm]{$\xi_2$};
\draw [decorate,decoration={brace,amplitude=5pt,mirror}]([yshift=1mm]B) -- ([yshift=1mm]C) node[midway,yshift=4mm]{$\lambda$};
\MarkRightAngle[size=6pt](C,B,A);
\end{tikzpicture}
\caption{The vectors $\xi_0$, $\xi_1$ and $\xi_2$.}
\end{center}
\end{figure}
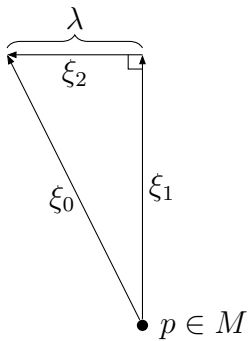

Let us define $q_j\in\p M$ to be the point at which the line $\gamma_{p,\xi_j}$ intersects $\p M$, defined such that the vector $p-q_j$ is a positive multiple of $\xi_j$. Then, fixing $N\in\N$ sufficiently large, we define the amplitude functions $a_k^{(j)}$ as follows. We first construct $a_0^{(j)}$ by letting $\xi=\xi_j$ in expression \eqref{amp0}, and define the remaining amplitudes $a_1^{(j)},\dots,a_N^{(j)}$ by letting $\xi=\xi_j$ in \eqref{ampsk} for all $y$ in the hyperplane $\Sigma_{q_j,\xi_j}$. Then, for $\tau>0$, we can define the approximate geometric optics solutions
\begin{equation*}
U_j(t,x)=e^{i(\tau\xi_j\cdot x-\tau^2|\xi_j|^2t)}\Bigg(\sum_{k=0}^N\tau^{-k}a_k^{(j)}(t,x)\Bigg).\\
%&U_0(t,x)=e^{i(\tau\xi_0\cdot x-\tau^2t)}\Bigg(\sum_{j=0}^N\tau^{-j}a_j^{(0)}(t,x)\Bigg),\\
%&U_1(t,x)=e^{i(\tau\xi_1\cdot x-(1-\lambda^2)\tau^2t)}\Bigg(\sum_{j=0}^N\tau^{-j}a_j^{(1)}(t,x)\Bigg),\\
%&U_2(t,x)=e^{i(\tau\xi_2\cdot x-\lambda^2\tau^2t)}\Bigg(\sum_{j=0}^N\tau^{-j}a_j^{(2)}(t,x)\Bigg),
\end{equation*} Further, for each $U_j$, we can define the corresponding exact solution $u_j=U_j+R_{\tau,j}$, where $R_{\tau,j}$ satisfies \eqref{remainder}. Letting $\eta\in C_0^\infty(M)$ satisfy $\eta=1$ in $M\setminus\Omega$, we define the source term $f_j=\big(i\p_t+\Delta+V\big)(\eta u_j)$. Then the function $\eta u_j$ solves
\begin{align}
\begin{cases}
i \p_t \mathcal{U}_j + \Delta \mathcal{U}_j + V \mathcal{U}_j = f_j & \text{on $(0,T) \times M$,}
\\
\mathcal{U}_j|_{x \in \p M} = 0,
\\
\mathcal{U}_j|_{t = 0} = 0.
\end{cases}
    \end{align}
We note that the sources $f_j$ are supported in $(0,T)\times\Omega$. Further, it can be shown that $f_j\in H^{2\kappa}_0$ is determined by the source-to-solution map $L_{\beta,V}$, up to any error $O(\tau^{-K})$ in the $H^{2\kappa}$-norm. To see this, let us first recall that $[(i\p_t+\Delta+V),\eta]=2\nabla\eta\cdot\nabla+\Delta\eta,$  this implies  that $f_j$ is given by the expression\[f_j=2\nabla\eta\cdot\nabla U_j+\Delta\eta U_j+2\nabla\eta\cdot\nabla R_j+\Delta\eta R_j.\]
We note that the first two terms on the right-hand side are uniquely determined by $L_{\beta,V}$ as a result of Lemma \ref{e_lemma}. On the other hand, we can apply the estimate \eqref{e_NRG_rem} to the last two terms to conclude that they are $O(\tau^{-N+4\kappa+2})$ in the $H^{2\kappa}$-norm.

Then, letting $\e_1,\e_2>0$ be small, we set $\e=(\e_1,\e_2)$ and define the source term $f=\e_1f_1+\e_2f_2$. For small enough $\varepsilon$, it holds that $f\in\mathcal{H}$, and we observe that\[-\frac{1}{2}\p_{\e_1}\p_{\e_2}L_{\beta,V}f|_{\e=0}=w|_{(0,T)\times\Omega},\]with $w$ the solution of the linear Schr\"odinger equation\begin{align}\begin{cases}i\p_tw+\Delta w+Vw=\beta\mathcal{U}_1\mathcal{U}_2\\w|_{x \in \p M} = 0\\w|_{t=0} = 0.\end{cases}\end{align}
Then it follows that\begin{equation}\label{e_before}-\frac{1}{2}\int_{(0,T)\times\Omega}\p_{\e_1}\p_{\e_2}L_{\beta,V}f|_{\e=0}\ \overline{f_0}\ \df x\,\df t = \int_{(0,T)\times M}w\ \overline{(i\p_t+\Delta+V)\mathcal{U}_0}\ \df x\, \df t.\end{equation}
Recall that $L_{\beta,V}$ is a continuous map from $\mathcal{H}$ into $H^{2\kappa}_0$, and that $f\in\mathcal{H}$ is determined by $L_{\beta,V}$ up to $O(\tau^{-N+4\kappa+2})$. Thus, the map $L_{\beta,V}$ determines the left-hand side of \eqref{e_before}, up to this error. We now integrate the right-hand side of (\ref{e_before}) by parts, and observe that it is given by\begin{equation}\label{e_after}\int_{(0,T)\times M}\beta\overline{\mathcal{U}_0}\mathcal{U}_1\mathcal{U}_2\, \df x\, \df t.\end{equation}We would like to approximate $\mathcal{U}_j$ in this integral by $\eta U_j$. Therefore, let $\kappa$ be large enough that $H^{2\kappa}_0$ is a Banach algebra, and let $N\geq4\kappa+4$. We see that, up to an error $O(\tau^{-2})$, the integral \eqref{e_after} coincides with the integral\[\int_{(0,T)\times M}\beta\eta^3\,\overline{U_0} U_1 U_2\, \df x\, \df t.\]
 Since $\phi(t)$ appearing in the definition of $a_0^{(j)}$ is arbitrary, it follows that $L_{\beta,V}$ determines for all $t\in(0,T)$ the integral
 \begin{equation}\label{e_I}\mathcal{I}=\int_M\beta(t,\cdot)\eta^3\,\overline{U_0(t,\cdot)}U_1(t,\cdot)U_2(t,\cdot)\,\df x,\end{equation}up to a small error $O(\tau^{-2})$. Henceforth, we shall supress this $t$-dependence in our notation. We now expand the above integral \eqref{e_I} in powers of $\tau$ as\[\mathcal{I}=\mathcal{I}_0+\mathcal{I}_1\tau^{-1}+O(\tau^{-2}).\]

Observe that the phases of the approximate geometric optics solutions cancel each other in the product $\overline{U_0}U_1U_2$, and further that $\eta=1$ in $\supp\big(\overline{U_0}U_1U_2\big)$. Therefore, it follows that the source to solution map $L_{\beta,V}$ determines the integrals\[\mathcal{I}_0=\int_M\beta a_0^{(0)}a_0^{(1)}a_0^{(2)}\, \df x, \quad \mbox{and}\quad \mathcal{I}_1=\sum_{|e|=1}\int_M\beta a_{e_0}^{(0)}a_{e_1}^{(1)}a_{e_2}^{(2)}\,\df x,\] where $e$ is a multi-index, $e=(e_0,e_1,e_2)\in(\N \cup\{0\})^3$. Then, letting $\chi_\delta$ in the definition \eqref{amp0} of $a_0^{(j)}$ converge to the indicator function of the interval $(-\delta,\delta)$, we obtain\[\mathcal{I}_0=\int_{P_\delta}\beta(x)\, \df x,\]where $P_\delta$ is a small neighbourhood of $p$ contained within a ball of radius $\delta$. Thus, by letting $\delta\rightarrow0$, we recover the quantity\[\lim_{\delta\rightarrow0}\frac{1}{|P_\delta|}\mathcal{I}_0=\beta(p).\]Since the choice of $p\in M\setminus\Omega$ was arbitrary, we recover the function $\beta(t,x)$ everywhere. 

To recover the potential $V$, we now consider the integral $\mathcal{I}_1$. We recall from (\ref{ampsk}) that $a_1^{(j)}$ is of the form $a_1^{(j)}=b_1^{(j)}+c_1^{(j)}$, where it holds that\begin{equation}\begin{gathered}\label{e_split}b_1^{(j)}(s\xi_j+y)=\frac{i}{2}\int_0^s\Big[(i\p_t+\Delta)a_0^{(j)}\Big](\tilde{s}\xi_j+y)\df\tilde{s},\\c_1^{(j)}(s\xi_j+y)=\frac{i}{2}\int_0^s\Big[Va_0^{(j)}\Big](\tilde{s}\xi_j+y)\df\tilde{s},\end{gathered}\end{equation}
for all $y$ in the hyperplane $\Sigma_{q_j,\xi_j}=\{x\in\R^n:\xi_j\cdot(x-q_j)\}$.
In particular, since $a_0^{(j)}$ is independent of $V$, so is $b_1^{(j)}$. Thus $L_{\beta,V}$ determines the quantity\[\mathcal{J}=\int_Mc_1^{(0)}a_0^{(1)}a_0^{(2)}\beta\, \df x+\int_Ma_0^{(0)}c_1^{(1)}a_0^{(2)}\beta\, \df x+\int_Ma_0^{(0)}a_0^{(1)}c_1^{(2)}\beta\, \df x.\]Then, by letting $\chi_\delta$ in the definition of $a_0^{(j)}$ converge to the indicator function of the interval $(-\delta,\delta)$, we deduce that $L_{\beta,V}$ determines the quantity\[\mathcal{J}_\delta=\sum_{j=0}^2\int_{P_\delta}\beta c_j\,\df x,\quad c_j(s\xi_j+y)=\frac{i}{2}\int_0^sV(\tilde{s}\xi_j+y)\, \df \tilde{s},\]
for a small neighbourhood $P_\delta$ of $p$. Then, since $\beta(p)$ is known and non-zero for almost all $p$, we can let $\delta\rightarrow0$ to recover the quantity\begin{equation}\label{2-to-1}\frac{1}{\beta(p)}\lim_{\delta\rightarrow0}\frac{1}{|P_\delta|}\mathcal{J}_\delta=\sum_{j=0}^2c_j(p).\end{equation}
It remains only to show that $V$ can be recovered from \eqref{2-to-1}. To this end, let $\widehat{\xi}$ denote the vector of unit length in the direction of $\xi_2$, so that $\xi_2=\lambda\widehat{\xi}$, and let $s_0\in\R$ be such that $\gamma_{q_2,\widehat{\xi}}(s_0)=p$. Then, it holds that\[-2ic_2(p)=\int_0^{s_0/\lambda}V(\tilde{s}\xi_2+q_2)\df\tilde{s}=\lambda^{-1}\int_0^{s_0}V(s\widehat{\xi}+q_2)\,\df s,\]and we can similarly check that $c_j(p)=O(1)$ as $\lambda\rightarrow0$ for $j\neq2$. Therefore, we have shown that we can recover from $L_{\beta,V}$ the quantity\[-2i\lim_{\lambda\rightarrow0}\Bigg(\lambda\sum_{j=0}^2c_j(p)\Bigg)=\int_0^{s_0}V(s\widehat{\xi}+q_2)\, \df s.\]But this is precisely the truncated ray transform of $V$, which we can differentiate with respect to $s_0$ to recover $V$.
\section{The Gross-Pitaevskii Equation}\label{e_s4}
We again fix $T>0$ and consider the case where $M$ is a Euclidean domain of $\R^n$ and $\Omega$ is a neighbourhood of $\p M$. For a potential $V\in C^\infty_0((0,T)\times M\setminus\Omega)$, and a coupling coefficient $\beta\in C^\infty_0((0,T)\times M\setminus\Omega)$ such that $\beta$ is non-zero almost everywhere in $\supp(V)$, we consider the problem of finding $u$ which, for a given source term $f$, solves the Gross-Pitaevskii equation
\begin{align}\label{GP}
\begin{cases}
(i \p_t + \Delta + V + \beta |u|^2) u = f & \text{on $(0,T) \times M$,}
\\
u|_{x \in \p M} = 0,
\\
u|_{t = 0} = 0.
\end{cases}
    \end{align}
We now let  $L_{\beta,V}$ denote the source-to-solution map, defined by\[L_{\beta,V}f=u\vert_{(0,T)\times\Omega},\]where $f$ is supported in $(0,T)\times\Omega$, and $u$ solves the Gross-Pitaevskii equation \eqref{GP} for the chosen source term $f$. We note that this $L_{\beta,V}$ is also a smooth map from some $\mathcal{H}$ to $H^{2\kappa}_0$ for large enough $\kappa$, by the same argument used for the non-linear Schr\"odinger equation in Section \ref{e_s2}. We now present the proof of Theorem \ref{th:GP}.

\begin{proof}[Proof of Theorem \ref{th:GP}]
    Let us now fix some $p\in M\setminus\Omega$, and some small $\lambda>0$. Let $\xi_0,\cdots,\xi_3$ be vectors in $\R^n$ which depend on $\lambda$, where the direction of $\xi_3$ is the same for all $\lambda>0$. Observe that it is possible to choose such vectors so that they satisfy\[\xi_0+\xi_1=\xi_2+\xi_3,\]with the $\xi_j$ pairwise non-colinear, and such that they satisfy the conditions\[|\xi_0|^2=1/2,\quad|\xi_1|^2=1/2,\quad|\xi_2|=1-\lambda^2,\quad|\xi_3|^2=\lambda^2.\]

\begin{figure}[ht]
\begin{center}
\begin{tikzpicture}[scale=0.9]
\coordinate [circle,fill,inner sep=1.5pt,label={[label distance=0.1mm]0:$p\in M$}](A) at (0,0);
\coordinate (B) at (0,4);
\coordinate (C) at (-2,4);
\coordinate (M) at ($(A)!0.5!(C)$);
\coordinate (P) at ($(M)!1!90:(C)$);
\draw [dashed](A) -- (C);
\draw [->,>={Triangle[scale width=0.6]}](A) -- (P) node[midway,xshift=-0.4mm,yshift=-2.6mm]{$\xi_0$};
\draw [->,>={Triangle[scale width=0.6]}](P) -- (C) node[midway,yshift=0.6mm,xshift=-2.2mm]{$\xi_1$};
\draw [->,>={Triangle[scale width=0.6]}](A) -- (B) node[midway,xshift=2.6mm]{$\xi_2$};
\draw [->,>={Triangle[scale width=0.6]}](B) -- (C) node[midway,yshift=-2.5mm]{$\xi_3$};
\draw [decorate,decoration={brace,amplitude=5pt,mirror}]([yshift=1mm]B) -- ([yshift=1mm]C) node[midway,yshift=4mm]{$\lambda$};
\MarkRightAngle[size=6pt](C,B,A);
\MarkRightAngle[size=6pt](A,P,C);
\end{tikzpicture}
\caption{The vectors $\xi_0$, $\xi_1$, $\xi_2$ and $\xi_3$.}
\end{center}
\end{figure}
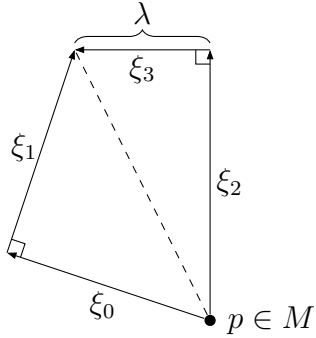
Let us define $q_j\in\p M$ to be the initial point of the line $\gamma_{p,\xi_j}$, as we did previously. Then for $\tau>0$ set
\begin{align*}\label{e_GPGO}U_j(t,x)=e^{i(\tau\xi_j\cdot x-\tau^2|\xi_j|^2t)}\Bigg(\sum_{k=0}^N\tau^{-k}a_k^{(j)}(t,x)\Bigg)
%&U_0(t,x)=e^{i(\tau\xi_0\cdot x-\frac{1}{2}\tau^2t)}\Bigg(\sum_{j=0}^N\tau^{-j}a_j^{(0)}(t,x)\Bigg),\\
%&U_1(t,x)=e^{i(\tau\xi_1\cdot x-\frac{1}{2}\tau^2t)}\Bigg(\sum_{j=0}^N\tau^{-j}a_j^{(1)}(t,x)\Bigg),\\
%&U_2(t,x)=e^{i(\tau\xi_2\cdot x-(1-\lambda^2)\tau^2t)}\Bigg(\sum_{j=0}^N\tau^{-j}a_j^{(2)}(t,x)\Bigg),\\
%&U_3(t,x)=e^{i(\tau\xi_3\cdot x-\lambda^2\tau^2t)}\Bigg(\sum_{j=0}^N\tau^{-j}a_j^{(3)}(t,x)\Bigg),
\end{align*}
where, for each $j$, we construct $a_0^{(j)}$ by letting $\xi=\xi_j$ in \eqref{amp0}, and define the remaining amplitudes $a_k^{(j)}$ by taking $\xi=\xi_j$ in \eqref{ampsk}. For each $U_j$, we once again define the corresponding exact solution $u_j=U_j+R_{\tau,j}$ as in \eqref{remainder}. Letting $\eta\in C_0^\infty(M)$ satisfy $\eta=1$ in $M\setminus\Omega$, we choose $f_j=\big(i\p_t+\Delta+V\big)(\eta u_j)$, and observe that $\eta u_j$ solves
\begin{align}
\begin{cases}
i \p_t \mathcal{U}_j + \Delta \mathcal{U}_j + V \mathcal{U}_j = f_j & \text{on $(0,T) \times M$,}
\\
\mathcal{U}_j|_{x \in \p M} = 0,
\\
\mathcal{U}_j|_{t = 0} = \mathcal{U}_j|_{t = T} = 0.
\end{cases}
    \end{align}
Then, letting $\e_1,\e_2,\e_3>0$ be small, we set $\e=(\e_1,\e_2,\e_3)$ and define the source term $f=\e_1f_1+\e_2f_2+\e_3f_3$. For small enough $\varepsilon$, it holds that $f\in\mathcal{H}$, and by linearizing the equation \eqref{GP} we deduce that\[-\frac{1}{6}\p_{\e_1}\p_{\e_2}\p_{\e_3}L_{\beta,V}f|_{\e=0}=w|_{(0,T)\times\Omega},\]where $w$ solves the linear Schr\"odinger equation.
\begin{align}
\begin{cases}
i\p_tw+\Delta w+Vw=\beta\Big(\overline{\mathcal{U}_1}\mathcal{U}_2\mathcal{U}_3+\mathcal{U}_1\overline{\mathcal{U}_2}\mathcal{U}_3+\mathcal{U}_1\mathcal{U}_2\overline{\mathcal{U}_3}\Big)\\w|_{x \in \p M} = 0\\w|_{t=0} = 0.
\end{cases}
\end{align}

Therefore, it holds that\[-\frac{1}{2}\int_{(0,T)\times\Omega}\p_{\e_1}\p_{\e_2}L_{\beta,V}f|_{\e=0}\ \overline{f_0}\, \df x\,\df t = \int_{(0,T)\times M}w\ \overline{(i\p_t+\Delta+V)\mathcal{U}_0}\, \df x\,\df t\]We can integrate by parts to see that the right-hand side of the above is given by\[I=\int_{(0,T)\times M}\beta\overline{\mathcal{U}_0}\overline{\mathcal{U}_1}\mathcal{U}_2\mathcal{U}_3\, \df x\,\df t+\int_{(0,T)\times M}\beta\overline{\mathcal{U}_0}\mathcal{U}_1\overline{\mathcal{U}_2}\mathcal{U}_3\, \df x\,\df t+\int_{(0,T)\times M}\beta\overline{\mathcal{U}_0}\mathcal{U}_1\mathcal{U}_2\overline{\mathcal{U}_3}\, \df x\,\df t.\]

Note that the phases of the geometric optics solutions cancel only in the first integral appearing in $I$. Therefore, since $\beta,a_j^{(k)}\in C^\infty_0((0,T)\times M)$, an integration by parts tells us that, choosing large enough $N$ in \eqref{e_GPGO}, for any $K\in\N$ we have\[I=\int_{(0,T)\times M}\beta\overline{\mathcal{U}_0}\overline{\mathcal{U}_1}\mathcal{U}_2\mathcal{U}_3\, \df x\,\df t+O(\tau^{-K}).\] Then, choosing $K\geq2$, we can expand the integral $I$ as\[I=I_0+I_1\tau^{-1}+O(\tau^{-2}),\] and arguing as we did for the non-linear Schr\"odinger equation in Section \ref{e_s3}, we deduce that the source-to-solution map $L_{\beta,V}$ determines for all $t\in(0,T)$ the integrals\[\begin{split}I_0=\int_M\beta a_0^{(0)}a_0^{(1)}a_0^{(2)}a_0^{(3)}\, \df x\quad\textrm{and}\quad I_1=\sum_{|e|=1}\int_M\beta a_{e_0}^{(0)}a_{e_1}^{(1)}a_{e_2}^{(2)}a_{e_3}^{(3)}\, \df x,\end{split}\]where $e$ is a multi-index, $e=(e_0,e_1,e_2,e_3)\in(\N\cup \{0\})^4$.
Next, letting $\chi_\delta$ in the definition of $a_0^{(j)}$ converge to the indicator function of the interval $(-\delta,\delta)$, we obtain
$I_0=\int_{P_\delta}\beta(x)dx$, where $P_\delta$ is a small neighbourhood of $p$ contained in a ball of radius $\delta$. Taking the limit as $\delta\rightarrow0$, we can recover the quantity $\lim_{\delta\rightarrow0}\frac{1}{|P_\delta|}I_0=\beta(p).$ Since the choice of $p\in M\setminus\Omega$ was arbitrary, we thus recover the function $\beta(t,x)$ everywhere.

To recover the potential $V$, we consider the integral $I_1$. Arguing as we did for the non-linear Schr\"odinger equation in Section \ref{e_s3}, we conclude that $L_{\beta,V}$ determines the quantity $\sum_{j=0}^nc_j(p)$ where we have\[c_j(s\xi_j+y)=\frac{i}{2}\int_0^sV(\tilde{s}\xi_j+y)d\tilde{s}\] for all $y\in\Sigma_{q_j,\xi_j}$. We now let $\widehat{\xi}$ denote the vector of unit length in the direction $\xi_3$, so that we have $\xi_3=\lambda\widehat{\xi}$, and let $s_0\in\R$ be such that $\gamma_{q_3,\widehat{\xi}}(s_0)=p$. Then, we can check that\[-2ic_3(p)=\lambda^{-1}\int_0^{s_0}V(s\widehat{\xi}+q_3)ds\]and similarly that $c_j(p)=O(1)$ as $\lambda\rightarrow0$ for $j\neq3$. Thus, we can recover from $L_{\beta,V}$ the quantity\[-2i\lim_{\lambda\rightarrow0}\Big(\lambda\sum_{j=0}^3c_j(p)\Big)=\int_0^{s_0}V(s\widehat{\xi}+q_3)ds.\]But this is just the truncated ray-transform of $V$, and we can differentiate with respect to $s_0$ in order to recover $V$.
\end{proof}
\section{The geometric case}\label{e_s5}
We again fix $T>0$ and let  $(M,g)$ be a compact Riemannian manifold  and $\Omega$ is a neighbourhood of $\p M$. For a potential $V\in C^\infty_0((0,T)\times M\setminus\Omega)$, and a coupling coefficient $\beta\in C^\infty_0((0,T)\times M\setminus\Omega)$ such that $\beta$ is non-zero almost everywhere in $\supp(V)$.  Let us recall the non-linear Schr\"odinger equation on $(M,g)$.
\begin{align}\label{IBVP_manifold}
\begin{cases}
(i \p_t + \Delta_g + V)u + \beta u^2= f & \text{on $(0,T) \times M$,}
\\
u|_{x \in \p M} = 0,
\\
u|_{t = 0} = 0.
\end{cases}
    \end{align}
Let us also recall the source-to-solution map as follows:
\[   L_{\beta,V}f = u|_{(0,T)\times \Omega},\]
where $f$ is supported in $(0,T)\times\Omega$, and $u$ solves the equation \eqref{IBVP_manifold} for the chosen source term $f$. We note that this $L_{\beta,V}$ is also a smooth map from some $\mathcal{H}$ to $H^{2\kappa}_{00}$ for large enough $\kappa$, by the same argument used for the non-linear Schr\"odinger equation in the Euclidean setting; see Section \ref{e_s2}. We now present the proof of Theorem \ref{th:main_result}. 

The rest of the section is organised as follows. In subsection \ref{subsec:Gaussian_beam} we give the Gaussian beam construction of the linearised Schr\"odinger equation on $(M,g)$. We then utilize those solutions in subsection \ref{subsec:boundary_source_determine} to determine amplitudes up to small error from the source-to-solution map. Finally in subsection \ref{subsec:main_result_thm_1.2} we present the proof of Theorem \ref{th:main_result}.   

%In this section we consider the problem on Riemannian manifolds.

\subsection{Gaussian beam construction}\label{subsec:Gaussian_beam}
In this section we present the construction of special solutions of  $  (\I \PD_t +\Delta_g +V)u=0$ on general manifolds. We closely follow the ideas from \cite{Dos_Jems}*{Section 3}. These special solutions concentrate near a given geodesic and they are usually called Gaussian beams in the literature. We next embed the manifold $ (M,g)$  into a  closed manifold $(N,g)$. To this end, we recall some results without their proofs.

\begin{lemma}[\cite{ma_sahoo_salo_anisotropic_high_fre}*{Lemma 6.7}]
	Let $(N,g)$ be a closed Riemannian manifold and $\gamma:(a,b)\rightarrow N$ be a unit speed non-trapping geodesic. Then 
$\gamma$ intersects itself only finitely many times.
\end{lemma}

\begin{lemma}[\cite{Dos_Jems}*{Lemma 3.4}]
	Let $F$ be a $C^{\infty}$ diffeomorphism of $(a,b) \times \{0\} \subset \Rn$ to a smooth $n$ dimensional manifold such that $F|_{(a,b) \times \{0\}}$ is injective and $DF(t,0)$ is invertible for $ t\in (a,b)$.
	If $[a_0,b_0]$ is a closed subinterval of $(a,b)$, then $F$ is a $C^{\infty}$ diffeomorphism in some neighbourhood of $ [a_0,b_0]\times \{0\} \in \Rn$. 
\end{lemma}

\begin{lemma}[\cite{Dos_Jems}*{Lemma 3.5}] \label{lem:Fermi}
	Let $(N,g)$ be a closed Riemannian manifold. Let $ \gamma$ be a unit speed geodesic with no loops. 
	Given a closed subinterval $[a_0,b_0]$ of $(a,b)$ such that $\gamma|_{[a_0,b_0]}$ intersects itself only at finitely many times $\{r_j\}_{j = 1}^{N}$ with $r_0 = a_0 < r_1 < \cdots < r_{N} < r_{N+1} = b_0$,
	there exists an open cover $\{(U_j,\phi_j)\}_{j=0}^{N+1}$ of $\gamma[a_0,b_0]$ consisting of coordinate neighbourhoods having the following properties:
	\begin{itemize}
		\item [1.] $\phi_j(U_j)= I_{j}\times B$ where
		$I_0 = (a, r_0 - \epsilon)$,
		$I_j = (r_{j-1} - 2\epsilon, r_j - \epsilon)$ with $j = 1,\cdots,N$,
		$I_{N+1} = (r_N - 2\epsilon, b)$,
		and $B=B(0,\delta)$ is an open ball in $\mathbb{R}^{n-1}$, where $\delta$ can be taken arbitrarily  small.
		\item [2.] $\phi_j(\gamma)=(r,0,\cdots,0)$ for $ r\in I_j$.
		\item [3.] $r_j\in I_j$ and $ \Bar{I}_j\cap \Bar{I}_k$ is empty unless $ |j-k|\le 1$.
		\item [4.] $\phi_j=\phi_k$ on $ \phi^{-1}_j\left((I_j\cap I_k) \times B\right)$.
	\end{itemize}
	Further, the metric in these coordinates satisfies $g^{jk}|_{\gamma(t)} = \delta^{jk}$, $\partial_i g^{jk}|_{\gamma(t)} = 0$.
\end{lemma}

We now present the construction of special solutions of 
\begin{align}\label{!LS}
    \begin{cases}
         (\I \PD_t +\Delta_g +V)u=0 \quad \mbox{in $M$}\\
         u(0,x)=0 \quad \mbox{in $M$} .
    \end{cases}
\end{align} known as Gaussian beams along a geodesic $\gamma$. We first assume that  $ \gamma$ does not intersect itself. Hence, we can have one coordinate chart of $\gamma$ denoted as $ (W,\phi)$ and $(r,y)$ be the local coordinate within this chart.
Let us consider the approximate Gaussian beam in the following form.\[U_\tau(t,r,y)=e^{i(\tau\psi(r,y)-\tau^2 t)}a(t,r,y),\]in the geodesic coordinates $(r,y)$. The phase function $\psi\in C^\infty(M)$ and amplitude $a_\tau\in C^\infty_0((0,T)\times M)$ are to be determined below. We write $P_V = \I \PD_t +\Delta_g +V$ and compute
\begin{equation}\label{!SU}P_V \big(e^{\I(\tau\psi-\tau^2 t)}a\big)=e^{\I(\tau\psi-\tau^2 t)}\Big(\tau^2(\mathcal{E}\psi)a+2i\tau\mathcal{T}a+P_Va\Big),\end{equation}where the operators $\mathcal{E},\mathcal{T}:C^\infty((0,T)\times M)\rightarrow C^\infty((0,T)\times M)$ are defined by
\[\mathcal{E}\psi:=1-\lr d\psi,d\psi \rn_g,\quad \mathcal{T}a:=\lr d\psi,da\rn_g+\frac{1}{2}(\Delta_g\psi)a.\]
We wish to solve the eikonal equation $\mathcal{E} \psi =0$ and the transport equation $\mathcal{T}a=0$ to determine the phase function and the amplitude on $(M,g)$. On simple manifolds one can construct these functions explicitly because the exponential map  is a global diffeomorphism, see \cite{DOS}. Hence, one can work with a global coordinate system to write the solutions globally on $(M,g)$, however this is not the case for non-simple manifolds. On a non-simple manifold, the idea is to prescribe the Taylor series of the functions $\psi$ and $a$ in a tubular neighbourhood of the geodesic $\gamma$. This can be obtained by imposing the following conditions on the eikonal and  transport equations
\begin{align}\label{!eikonal}
 \frac{\PD^\alpha}{\PD y^\alpha}(\mathcal{E}\psi)(r,0,\cdots,0)&=0\quad \forall r\in I,\\ \label{!transport}
 \frac{\PD^\alpha}{\PD y^\alpha}(\mathcal{T}a)(r,0,\cdots,0)&=0\quad \forall r\in I,   
\end{align}
where $ \alpha\in (\N\cup\{0\})^{n-1}$ is a multi-index with $ \abs{\alpha}\le N.$ 

\begin{remark}\label{rem_y_shift}
Below we will use the fact that if $a$ satisfies \eqref{!transport}
and $\beta \in \N^{n-1}$ then $\tilde a(r, y) = y^\beta a(r,y)$
also satisfies \eqref{!transport}.
\end{remark}

%We now focus on the construction of phase function $\psi$. %\subsubsection{The Phase Function}\label{e_s5.1.1}
Let us construct $\psi$ satisfying the eikonal equation \eqref{!eikonal}. We write %\begin{equation}\label{eq_eikonal}
$\psi=\sum_{j=0}^N\psi_j(r,y)$, where $\psi_j(r,y)$ is a homogeneous function of degree $j$ in the $y$ variable.
%\end{equation}
%In what follows, we write $\p_1,\p_r$ interchangeably for $\p_{z^1}$, and let $\p_j$ denote $\p_{z^j}=\p_{y^{j-1}}$ for $j=2,\cdots,n$. 
By taking $|\alpha|=0$ in \eqref{!eikonal}, we obtain the following equation on $\gamma$
\[\sum_{k,l=1}^ng^{kl}|_{\gamma}\p_k\psi\p_l\psi=1,\quad \forall \,\,r\in I.\] 
Using the fact that $g^{kl}|_\gamma=\delta^{kl}$, this becomes $\sum_{l=1}^n(\p_l\psi)^2=1.$
%\begin{equation}\label{!eik0}\end{equation}
Similarly, we can take $|\alpha|=1$ in (\ref{!eikonal}) and use the fact that $g^{kl}=\delta^{kl}$ and $\p_ig^{kl}=0$ on $\gamma$ to deduce that $\sum_{l=1}^n\p_{il}^2\psi\p_l\psi=0\quad\forall r\in I$
%\begin{equation}\label{!eik1}\end{equation}
on $\gamma$ for any $i=2,\cdots,n$. These equations are satisfied by setting \begin{equation}\label{!psi01}\psi_0=r\quad\textrm{and}\quad\psi_1=0.\end{equation}
We continue by noting that equation \eqref{!eikonal} with $|\alpha|=2$ is equivalent to\[\sum_{k,l=0}^n \Big( 2g^{kl}\p_{ijk}^3\psi\p_l\psi+2g^{kl}\p_{ik}^2\psi\p_{jl}^2\psi+\p_{ij}^2g^{kl}\p_k\psi\p_l\psi+4\p_ig^{kl}\p_{jk}^2\psi\p_l\psi\Big)=0,\]for any $i,j=2,\cdots,n$. Then, since it holds on $\gamma$ that $\p_ig^{kl}=0$, that $\p_j\psi=1$ if $j=1$ and $\p_j\psi=0$ otherwise, and that $\p_1\p_{j}\psi=0$ if $j\neq1$, we can rewrite the above equation on $\gamma$ as:
\begin{equation}\label{!ric}\p_{ij}^2g^{11}+2\p_{1ij}^3\psi+2\sum_{k=2}^n\p_{ki}^2\psi\p_{kj}^2\psi=0\quad\forall r\in I.
\end{equation}
In light of the above, we set $\psi_2(r,y)=\frac{1}{2}\sum_{ij=1}^{n-1}H_{ij}(r)y^iy^j$, where $H$ is a smooth, symmetric, complex matrix with positive-definite imaginary part, $i.e., \Im H(r)>0\ \forall r\in I. $ %\begin{equation}\label{!ImH}\Im H(r)>0\ \forall r\in I.\end{equation}
Notice that $\p_{y^i}\p_{y^j}\psi=H^{ij}$. Therefore, in order to satisfy equation \eqref{!ric}, we require that $H(r)$ solves the matrix Riccati equation
\begin{equation}\label{!riccati}\frac{d}{dr}H+H^2+D=0\quad\forall r\in I,\end{equation}where the matrix $D$ is given by $D_{ij}=\frac{1}{2}\p_{y^i}\p_{y^j}g^{11}$. This is an ODE for $H$ and it  can be solved using \cite{KKL_Inv_bound_spect_pr}*{Lemma 2.56}, which shows that there exists a solution $H$ of equation \eqref{!riccati} with the desired properties, and this is sufficient to determine $\psi_2$.

We next continue to find the functions $\psi_3,\cdots,\psi_N$ by solving equation (\ref{!eikonal}) for $|\alpha|=3,\cdots,N$ respectively. For example, let us give a brief summary in the case where $|\alpha|=3$. We let $\p_y^\alpha=\p_i\p_j\p_k$ for $j=2,\cdots,n$, so that (\ref{!eikonal}) yields\[2\sum_{k,l=0}^n\Big(g^{kl}\p_l\p_y^\alpha\psi\p_m\psi+g^{lm}\p_{ijl}^3\psi\p_{km}^2\psi+g^{lm}\p_{ikl}^3\psi\p_{jm}^2\psi+g^{lm}\p_{jkl}^3\psi\p_{im}^2\psi\Big)+F_\alpha=0,\]where $F_\alpha$ depends only on $g^{lm}$ and $\psi_j$ for $j\leq2$. Since it holds on $\gamma$ that\[\sum_{l,m}g^{lm}\p_l\p_y^\alpha\psi\p_m\psi=\p_r\p_y^\alpha\psi,\]we observe that the coefficients $\p_y^\alpha\psi$ satisfy a system of linear ODEs on $\gamma$ with the right-hand side depending on $\psi_j$ and $\p_r\psi_j$ for $j\leq2$. Thus, by prescribing some initial condition at $r=r_0$, we can solve these systems uniquely to obtain $\psi_3$. The remaining polynomials $\psi_j$ of higher degree can then be constructed in the same manner. This completes the construction  of  phase function $\psi$.

We now focus on the construction of amplitudes. We start with writing $a$ as
\begin{equation}\begin{split}\label{eq transport}
%&\psi=\sum_{j=0}^N\psi_j(r,y),\\%\quad\textrm{and}\quad 
a(t,r,y)=\phi(t)\chi\Big(\frac{|y|}{\delta'}\Big)\sum_{k=0}^N\tau^{-k}a_k(t,r,y),\quad a_k(t,r,y)=\sum_{j=0}^N\,a_{k,j}(t,r,y),\end{split}\end{equation}where for $k,j=0,\cdots,N$, it holds that $a_{k,j}$ are $j$-th degree homogeneous polynomials in the $y$-variable, $\chi\in C^\infty_0(\R)$ satisfies $\chi(s)=1$ for $|s|\leq\frac{1}{4}$ and $\chi(s)=0$ for $|s|\geq\frac{1}{2}$, and $\phi\in C^\infty_0(0,T)$ is a smooth cutoff.
%We now turn to constructing the amplitude functions $v_k$, beginning with the leading amplitude $v_0$. 
%Taking $|\alpha|=0$ in \eqref{!transport}, and recalling the definition of $\mathcal{T}$, we obtain on $\gamma$ the equation\[\sum_{k,l=1}^n\big(g^{kl}\p_k\psi\p_lv_0\big)+\frac{1}{2}\Delta\psi a=0\quad\forall r\in I.\]
Recalling the definition of $\mathcal{T}a$, we see that $a_0$ satisfies 
\begin{align}\label{eq_transport_v0}
   \PD^{\alpha}_{y} \left( \sum_{k,l=1}^n\big(g^{kl}\p_k\psi\p_la_0\big)+\frac{1}{2}\Delta_g\psi a_0\right)=0\quad \mbox{on $\gamma$,}
\end{align}
where $ \alpha\in (\N\cup\{0\})^{n-1}$ is a multi-index with $ \abs{\alpha}\le N.$ 
By Lemma \ref{lem:Fermi} we have  that $g^{kl}|_{\gamma(r)}=\delta^{kl}$. Thus for $ |\alpha|=0$, this preceding equation reduces to $\frac{d}{dr}a_{0,0}+\frac{1}{2}\tr(H)a_{0,0}=0\quad\forall r\in I. $ This is an ODE  for $a_{0,0}$ along $ \gamma$, and its solution can written as:
\begin{equation}\label{!amp00}a_{0,0}=c_0e^{-\frac{1}{2}\int_{r_0}^r\tr(H)(s)ds},
\end{equation}
where $c_0$ is a constant chosen so that $a_{0,0}(r_0)=1$. This implies $c_0=1$.
%\HOX{Clarify the normalization. It is natural to set an initial condition at $r=r_0$. What is $p$? Is $c_0$ later denoted by $c_n$? The condition for the whole $v_0$ fixes initial conditions for $v_{0, j}$ that should be vanishing, I believe.}
%Further, it follows for this choice of $c_0$ that $v_0(p)=1$ as well.}
%\[\Delta\psi|_{\gamma}=\sum_{k,l=1}^n|g|^{-\frac{1}{2}}\p_k\big(|g|^{\frac{1}{2}}g^{kl}\p_l\psi\big)|_\gamma=\sum_{l=1}^n\p_{ll}^2\psi=\tr(H),\] 
%and therefore we obtain the equation \[\frac{d}{dr}v_{0,0}+\frac{1}{2}\tr(H)v_{0,0}=0\quad\forall r\in I.\]We can choose\begin{equation}\label{!amp00}v_{0,0}=c_0e^{-\frac{1}{2}\int_{r_0}^r\tr(H)(s)ds},\end{equation}where $c_0$ is a constant chosen so that $v_{0,0}(p)=1$. Further, it follows for this choice of $c_0$ that $v_0(p)=1$ as well.\\

The subsequent terms $a_{0,j}$ for $j=1,\cdots,N$ can be constructed by considering equation \eqref{eq_transport_v0} with $|\alpha|=j$, which reduces to solving linear first order ODEs on $\gamma$. Indeed, taking $|\alpha|=j$ in \eqref{eq_transport_v0} and recalling the definition of $\mathcal{T}$, we obtain an equation of the form
\begin{align}\label{eq_transport_v_0j}
\p_ra_{0,j}+\tr(H)a_{0,j}+\Pi_j=0\quad\forall r\in I\quad \mbox{and} \quad a_{0,j}(r_0)=0,
\end{align}
where $\Pi_j$ is a homogeneous polynomial of degree $j$ in the $y$-variable, whose coefficients depend only upon $\{a_{0,l}\}_{l=0}^{j-1}$ and $\{\psi_l\}_{l=0}^{j+2}$. 

To construct the subsequent amplitudes $a_k$, we need to solve the equation 
\begin{align}\label{eq_transport_v_k}
    \mathcal{T}a_k+ P_V a_{k-1}=0 \quad \mbox{ for $k\ge 1$ up to $N$-th order on $\gamma$.}
\end{align}
 This can be accomplished by much the same argument used for $a_0$, and so we omit the details and refer  \cites{Dos_Jems,Kenig_Salo_Survey}. However, let us establish an analogue of \eqref{e_split} for the approximate Gaussian beams. We consider equation \eqref{eq_transport_v_k} with $|\alpha|=0$ and $k=1$,  on $\gamma$. This gives \[2i\Big(\frac{d}{dr}a_{1,0}+\frac{1}{2}\tr(H)a_{1,0}\Big)=-(i\p_t+\Delta)a_{0,0}-Va _{0,0}.\]
Therefore, we may choose\begin{equation}\begin{aligned}\label{!split}
&a_{1,0}(r)=b_{1,0}(r)+c_{1,0}(r), \quad \mbox{where}\\
&b_{1,0}(r)=e^{-\frac{1}{2}\int_{r_0}^r\tr(H)(s)ds}\int_{r_0}^r\frac{i}{2}e^{-\frac{1}{2}\int_{r_0}^s\tr(H)d\tilde{s}}(i\p_t+\Delta)a_{0,0}(s) ds\\&c_{1,0}(r)=e^{-\frac{1}{2}\int_{r_0}^r\tr(H)(s)ds}\int_{r_0}^r\frac{i}{2}e^{-\frac{1}{2}\int_{r_0}^s\tr(H)d\tilde{s}}Va_{0,0}(s)ds.
\end{aligned}
\end{equation}
We recall that $a_0$ satisfies \eqref{eq_transport_v0}, whence it follows that $b_{1,0}$ does not depend on $V$. This completes the construction $a_\tau$.

The function $U_\tau=e^{i(\tau\psi-\tau^2t)}a_\tau$ is then an approximate Gaussian beam of order $N$ , provided that $\delta'$ is small enough and it satisfies 
\begin{enumerate}
\item The equations \eqref{!eikonal}, and \eqref{!transport} are satisfied.
\item $\Im(\psi)|_\gamma=0$, i.e. the imaginary part of the phase vanishes along $\gamma$.
\item There exists $c\geq0$ such that $\Im(\psi)(r,y)\geq c|y|^2$ for all $(r,y)\in \Phi(W)$.
\end{enumerate}
%It remains now to generalise the above work to geodesics with self-intersections, and to 

We now convert the approximate Gaussian beams into exact solutions of \eqref{!LS} via the addition of a suitable remainder term $i.e.,\, u= U_\tau + R_{\tau}$, where the remainder term satisfies 
\begin{align}\label{eq_rmdr}
    P_V R_\tau &= -P_V U_{\tau} \quad \mbox{in}\quad (0,T)\times M\\
 R_{\tau}|_{t=0}&=0.
\end{align}
Moreover, by Proposition \ref{!GBProp} one has 
\begin{align}\label{!GBRest}
    \nrm[H^s((0,T)\times M]{R_\tau} \lesssim \nrm[H^s((0,T)\times M]{P_V U_\tau}\lesssim \tau^{2s+\frac{3-N}{2}}. 
\end{align}

%In the light of above we have the following result.

% \frac{\PD^\alpha}{\PD y^\alpha}(\mathcal{E}\psi)(r,0,\cdots,0)=0,\quad\forall r\in I\end{equation}
% for $I$ , whenever $\alpha\in\mathbb{N}^{n-1}$ satisfies $|\alpha|\leq N$. We also require that the amplitudes $v_k$ satisfy the transport equations
% \begin{align}\label{!transport}\frac{\PD^\alpha}{\PD y^\alpha}(\mathcal{T}v_0)(r,0,\cdots,0)=0,\quad\forall r\in I\\\label{!transports}\frac{\PD^\alpha}{\PD y^\alpha}(2i\mathcal{T}v_k+P_Vv_{k-1})(r,0,\cdots,0)=0,\quad\forall r\in I,\end{align}
% whenever $k=1,\cdots,N$ and $\alpha\in\mathbb{N}^{n-1}$ satisfies $|\alpha|\leq N$.\\

% In particular, we require that $\psi$ satisfies the eikonal equation $\mathcal{E}\psi=0$ up to $N$th order on $\gamma$, that is\begin{equation}\label{!eikonal}
% \frac{\PD^\alpha}{\PD y^\alpha}(\mathcal{E}\psi)(r,0,\cdots,0)=0,\quad\forall r\in I\end{equation}
% for $I$ , whenever $\alpha\in\mathbb{N}^{n-1}$ satisfies $|\alpha|\leq N$. We also require that the amplitudes $v_k$ satisfy the transport equations
% \begin{align}\label{!transport}\frac{\PD^\alpha}{\PD y^\alpha}(\mathcal{T}v_0)(r,0,\cdots,0)=0,\quad\forall r\in I\\\label{!transports}\frac{\PD^\alpha}{\PD y^\alpha}(2i\mathcal{T}v_k+P_Vv_{k-1})(r,0,\cdots,0)=0,\quad\forall r\in I,\end{align}
% whenever $k=1,\cdots,N$ and $\alpha\in\mathbb{N}^{n-1}$ satisfies $|\alpha|\leq N$.\\

\begin{proposition}\label{!GBProp}
Let $U_\tau=e^{i(\tau\psi-\tau^2t)}a_\tau$ be an approximate Gaussian beam of order $N$. Then for $\tau\gg1$ we have\begin{equation}\label{!estimates}\norm{P_VU_\tau}_{H^s((0,T)\times W)}\lesssim\tau^{\frac{3-N}{2}+2s},\qquad\norm{U_\tau}_{H^s((0,T)\times W)}\lesssim \tau^{2s}.\end{equation}
\end{proposition}
\begin{proof}
Using the fact that $ \Im \psi$ satisfies $\Im(\psi)(r,y)\geq c|y|^2$, we obtain $|e^{i(\tau\psi-\tau^2 t)}|\leq Ce^{-\frac{1}{4}c\tau|y|^2}$ for sufficiently small $y$. Thus, since $a_\tau$ is a smooth, compactly supported function, the estimate\[\begin{split}\norm{U_\tau}_{H^s((0,T)\times W)}\leq& C\tau^{2s}\norm{U_\tau}_{L^2((0,T)\times W)}\\\lesssim&\tau^{2s}\norm{e^{-\frac{1}{4}c\tau|y|^2}\chi(|y|/\delta')}_{L^2((0,T)\times W)}=O(\tau^{2s}),\end{split}\] provided that $\delta'$ is small enough. By using equations \eqref{!eikonal}, \eqref{!transport} in the identity \eqref{!SU}, we deduce similarly that\begin{equation}\label{!1}|\p_z^\sigma P_VU_\tau|\lesssim\tau^{|\sigma|}|e^{i(\tau\psi-\tau^2 t)}|\Big(C_0\tau^2|y|^{N+1}+C_1\tau|y|^{N+1}+C_2\tau^{-N}\Big),\end{equation}for any $\sigma\in\N^{n}$. Further, we deduce that\begin{equation}\label{!2}|\p_t^m P_VU_\tau|\lesssim \tau^{2m}|e^{i(\tau\psi-\tau^2t)}|\Big(C_0\tau^2|y|^{N+1}+C_1\tau|y|^{N+1}+C_2\tau^{-N}\Big),\end{equation}by the same argument. Thus, it follows from \eqref{!1} and \eqref{!2} that\[\begin{split}\norm{P_VU_\tau}_{H^{s}((0,T)\times W)}\lesssim\tau^{2s}\norm{e^{-\frac{1}{4}c\tau|y|^2}(\tau^2|y|^{N+1}+\tau^{-N})}_{L^2((0,T)\times W)}= O(\tau^{\frac{3-N}{2}+2s}).\end{split}\]
\end{proof}
\begin{remark}\label{!conj}
  Observe that, if $U_\tau=e^{i(\tau\psi-\tau^2 t)}a_\tau$ is an approximate Gaussian beam of order $N$ along $\gamma$, then $\overline{U}_\tau=e^{-i(\tau\overline{\psi}-\tau^2t)}\overline{a}_\tau $ also satisfies the estimates of Proposition \ref{!GBProp}.
\end{remark}

\subsection{Determination of the boundary sources}\label{subsec:boundary_source_determine}
Let us recall that  \[ \Lc_V= \PD_{\epsilon} L_{V,\beta} (\epsilon f)|_{\epsilon=0}.\] Here we show that $\Lc_V$ determines the amplitudes up to a small error term. This is similar to Lemma \ref{e_lemma} in the Euclidean setting, however the proof is more involved in the geometric setting.
\begin{lemma}\label{lm:determining_amplitudes}
   Suppose that $ \Lc_{V_1}= \Lc_{V_2}$ in $ (0,T) \times \Omega$. Let $ U_{\tau, j}= e^{\I (\tau \psi -\tau^2 t)} a^j$ be an approximate Gaussian beam  solution of \eqref{!LS} concentrate near a geodesic $ \gamma$.  Then   $a^1_k=a^2_k$ on $  \Omega \cap \gamma $  up to higher order for all integers $k\ge 0$.
\end{lemma}
\begin{proof}

 The proof is based on induction on $k$.  Fix a geodesic $ \gamma$ and let $(r,y)$ be the Fermi coordinates in a tubular neighbourhood of $\gamma$. Let $\ell$ be the length of the geodesic $\gamma$ and we choose $0<r_1<r_2<\ell$ such that $(r,y)\in \Omega$ when $ r<r_1$ or $r>r_2$.  Now proving Lemma \ref{lm:determining_amplitudes} is equivalent to proving the following:  for any  integer $M\ge 0$ and for all integers $k\ge 0$
%\HOX{This should hold only for $r$ large. Fix the notation.}
\begin{align}\label{induction_argument}
  (\PD^{\alpha}_{y} a_k^1)(r, 0)=  (\PD^{\alpha}_{y} a_k^2)(r, 0) \quad \mbox{ for $r<r_1$ or $r>r_2$, and  multi-indices $|\alpha|=M$.}
 \end{align}
 %We will apply induction on $k$ as well as on $\alpha$.
 We divide the proof into several steps.\smallskip

%\HOX{Systematize the notation for amplitudes (should be aligned with the section on Gaussian beams. Mention that we hide the time variable to simplify the notation}
 \textbf{Step 1.} We show that 
 $ (\PD^{\alpha}_{y} a_0^1)(r, 0)=  (\PD^{\alpha}_{y} a_0^2)(r, 0) \quad\mbox{for any multi-index $|\alpha|=M$.}$\smallskip
 
 For $\abs{\alpha}=0$ we have that $ a^1_0-a^2_0$ satisfies the transport equation 
 \[   \mathcal{T}(a^1_0-a^2_0)=0 \quad \mbox{on $\gamma$ and  $(a^1_0-a^2_0)(0)=0$.}\]
As $a^1_0-a^2_0 = 0$ at $\gamma(0)$,
this implies $(a^1_0-a^2_0)(r,0)=0$ by unique solvability of ordinary differential equations. Suppose that 
\begin{align}\label{eq_induction_1}
    (\PD^{\alpha}_{y} a_0^1)(r, 0)=  (\PD^{\alpha}_{y} a_0^2)(r, 0) \quad\mbox{for any multi-index $|\alpha| \le M$. }
\end{align}
As $\mathcal Ta^j_0$ vanishes to a high order on $\gamma$,
it follows from \eqref{eq_induction_1} that
for a multi-index $|\alpha| = M + 1$ there holds
    \begin{align}
\mathcal T \p_y^\alpha (a^1_0-a^2_0) = 0 \quad \text{on $\gamma$}.
    \end{align}
As $\p_y^\alpha (a^1_0-a^2_0) = 0$ at $\gamma(0)$, we see that 
\eqref{eq_induction_1} holds for $|\alpha| = M + 1$. This completes the induction argument as well as \textbf{Step 1}.

 \smallskip
 
\textbf{Step 2.} Here  we show that \eqref{induction_argument} holds true for $k=K$, and $ |\alpha|=0$.\smallskip

We argue by induction and assume that \eqref{induction_argument} is true for all $k\le K-1$. 
Let $\eta \in C_c^{\infty}(M)$ such that $\eta=1$ in $M\setminus\Omega$.
We set $f_j=(\I \PD_t +\Delta_g +V_j)(\eta u_j)$, 
where $u_j$ solves 
\begin{equation}
    \begin{cases}
         (\I \PD_t +\Delta_g +V_j)u_j=0\quad &\mbox{in} \quad (0,T) \times M\\
       % u=0 \quad & \mbox{on} \quad \PD M\\
        u_j=0 \quad & \mbox{at} \quad t=0,
    \end{cases}
     \end{equation}
and coincides with the approximate Gaussian beam 
    \begin{align}
U_j= e^{\I (\tau \psi(r,y)- \tau^2 t)} a^j(t,r,y)= e^{\I (\tau \psi(r,y)- \tau^2 t)} \sum_{k=0}^N\, a_k^j (t,r,y)
    \end{align}
up to a small error of order $ \mathcal{O}(\tau^{-N})$ in the sense of $L^2$.
Clearly,  the function $ \eta u_j$ satisfies 
\begin{align}
    \begin{cases}
       (\I \PD_t +\Delta_g +V_j)u_j=f_j\quad &\mbox{in} \quad (0,T) \times M\\
        u_j=0 \quad & \mbox{on} \quad \PD M\\
        u_j=0 \quad & \mbox{at} \quad t=0,
    \end{cases}
\end{align}
with sources $f_j$ supported in $(0,T) \times \Omega$. 
Further, let us consider a solution of 
\begin{align}
    \begin{cases}
         (\I \PD_t +\Delta_g +V_2)w=0\quad &\mbox{in} \quad (0,T) \times M\\
       % u_j=0 \quad & \mbox{on} \quad \PD M\\
      w=0 \quad & \mbox{at} \quad t=T.
    \end{cases}
\end{align}
that coincides with the approximate Gaussian beam
    \begin{align}
e^{\I (\tau \psi(r,y)- \tau^2 t)} \sum_{k=0}^N\, w_k (t,r,y),
    \end{align}
up to a small error in the same sense.

Let $\tilde \eta \in C_c^{\infty}(M)$ such that $\tilde \eta=1$ in $M\setminus\Omega$.
We set $h=(\I \PD_t +\Delta_g +V_2)(\tilde \eta w)$. Analogously to \eqref{not_Alessandrini}, we have
\begin{align}\label{eq_1.7}
  \lr (\Lc_{V_j}-\Lc_{V_2})f_j,h\rn_{L^2((0,T)\times\Omega)}= \lr \eta u_j,h\rn_{L^2((0,T)\times\Omega)} - \lr f_j,\Tilde{\eta} w\rn_{L^2((0,T)\times\Omega)}.   
\end{align}
Subtracting \eqref{eq_1.7}
 for $j=1$ from \eqref{eq_1.7} for $j=2$ we obtain
 \begin{equation}\label{integral_identity}
     \begin{aligned}
  0=  &\agl[ e^{\I (\tau \psi(r,y)- \tau^2 t)} \eta\tau^{-K}( a^1_{K}-a^2_K), e^{\I (\tau \psi(r,y)- \tau^2 t)} 2\I \tau \mathcal{T} \Tilde{\eta} \,w_0+\mathcal{O}(1)]_{L^2((0,T)\times\Omega)}\\&- \agl[e^{\I (\tau \psi(r,y)- \tau^2 t)}(2\I \tau \mathcal{T} \eta+2\lr\nabla \eta,\nabla\rn_g+\Delta_g \eta) \tau^{-K} ( a^1_{K}-a^2_K), e^{\I (\tau \psi(r,y)- \tau^2 t)}\Tilde{\eta} w_0]_{L^2((0,T)\times\Omega)}\\&\quad-\agl[e^{\I (\tau \psi(r,y)- \tau^2 t)}(2\I \tau \mathcal{T} \eta+2\lr\nabla \eta,\nabla\rn_g+\Delta_g \eta) \tau^{-K} ( a^1_{K}-a^2_K),  \mathcal{O}(\tau^{-1})]_{L^2((0,T)\times\Omega)} \\
  &\qquad+ \mathcal{O}(\tau^{-K-1})+ \mathcal{O}(|y|^{\infty}).
 \end{aligned}
 \end{equation}
 We next consider the  coefficient of $\tau^{-K+1}$ from the above integral  identity and simplify. To this end, recall that we  choose $0 < r_1 < r_2 < \ell$ such that in Fermi coordinates $(r,y) \in \Omega$ when $r < r_1$ or $r > r_2$ and $\eta(r,y) = \tilde \eta(r,y) = 0$ when $r > \ell$. Let $I$ be the coefficient of $\tau^{-K+1}$ in the above expansion. Since $ \eta$ and $\tilde{\eta}$ are real valued and $\mathcal{T}(\cdot)$ is a complex valued function, this implies
%\HOX{Remove overlines. $T$ and $\eta$s are real. The notation $T_\psi$ does not seem to be defined.}
 \begin{align}
     I&= 2\I \int_0^T \int_{\Omega} e^{-2\tau \Im \psi} (\eta  \:\overline{\mathcal{T} \Tilde{\eta}}+ \Tilde{\eta}\mathcal{T}\eta) w_0 \, ( a^1_{K}-a^2_K) \df v_g\,\df t \nonumber\\
     &=2\,\I \int_0^T  \int_0^{r_1} \int_{\R^{n-1}} e^{-2\tau \Im \psi} (\eta  \:\overline{\mathcal{T} \Tilde{\eta}}+ \Tilde{\eta}\mathcal{T}\eta) w_0 \, ( a^1_{K}-a^2_K) \sqrt{|g|}\, \df r\,\df y\,\df t \notag\\& \quad+2\,\I \int_0^T  \int_{r_2}^{\ell}\int_{\R^{n-1}} e^{-2\tau \Im \psi} (\eta  \:\overline{\mathcal{T} \Tilde{\eta}}+ \Tilde{\eta}\mathcal{T}\eta) w_0 \, ( a^1_{K}-a^2_K) \sqrt{|g|}\, \df r\,\df y\,\df t.
 \end{align}
Since $ a_K^1= a_K^2$ up to higher order before the geodesic enters the domain $M\setminus\Omega$, that is when $r<r_1$, there holds \[ \int_0^T  \int_0^{r_1} \int_{\R^{n-1}} e^{-2\tau \Im \psi} (\eta  \:\overline{\mathcal{T} \Tilde{\eta}}+ \Tilde{\eta}\mathcal{T}\eta) w_0 \, ( a^1_{K}-a^2_K) \sqrt{|g|}\, \df r\,\df y\,\df t =\mathcal{O}(|y|^{\infty}).\] 
This further entails, up to $\mathcal{O}(|y|^{\infty})$
%\HOX{At this point $c_n$ hasn't been defined. It is defined below, though. The definition should be moved to the section on Gaussian beams and the notation should be systematized}
\begin{align*}
    I= 2\,\I \int_0^T  \int_{r_2}^{\ell}\int_{\R^{n-1}} e^{-2\tau \Im \psi} (\eta  \:\overline{\mathcal{T} \Tilde{\eta}}+ \Tilde{\eta}\mathcal{T}\eta) w_0 \, ( a^1_{K}-a^2_K) \sqrt{|g|}\, \df r\,\df y\,\df t.
\end{align*}
 Observe that, in the Fermi coordinates $(r,y)$  we have the following expression for $ \psi$:
\[\psi(r,y)= r +\frac{1}{2} H(r)y\cdot y+ \mathcal{O} (|y|^3)  \implies \mathrm{Im}(\psi) = \mathrm{Im}(\frac{1}{2} H(r)y\cdot y+ \mathcal{O} (|y|^3) ).\] 
 We next perform the change of variable $ y\rightarrow y/\sqrt{\tau}$ and obtain 
 \begin{align*}
     I=  &\frac{2\I}{\tau^{(n-1)/2}} \int_0^T  \int_{r_2}^{\ell}\int_{\R^{n-1}} e^{-\Im H(r) y\cdot y+\tau^{-1/2} \mathcal{O}(|y|^3)} (\eta  \:\overline{\mathcal{T} \Tilde{\eta}}+ \Tilde{\eta}\mathcal{T}\eta)(r,y/\sqrt{\tau})\\& \qquad\qquad\qquad\qquad\times w_0(t,r,y/\sqrt{\tau}) \, ( a^1_{K}-a^2_K)(r,y/\sqrt{\tau}) \sqrt{|g|}(r,y/\sqrt{\tau})\, \df y\, \df r\,\df t\\
     = &\frac{2\I}{\tau^{(n-1)/2} } \int_0^T  \int_{r_2}^{\ell} \frac{1}{\sqrt{\det\Im(H(r))}}\,\int_{\R^{n-1}} e^{- |y|^2-\tau^{-1/2} \mathcal{O}(|y|^3)} (\eta  \:\overline{\mathcal{T} \Tilde{\eta}}+ \Tilde{\eta}\mathcal{T}\eta)(r,(\cdot)/\sqrt{\tau})\\& \qquad\qquad\qquad\qquad\times w_0(t,r,\cdot/\sqrt{\tau}) \, ( a^1_{K}-a^2_K)(r,\cdot/\sqrt{\tau}) \sqrt{|g|}(r,\cdot/\sqrt{\tau})\, \df y\, \df r\,\df t.
 \end{align*}
We multiply \eqref{integral_identity} by $\tau^{-K+1}$ and  let $ \tau \rightarrow \infty$ to conclude
 \begin{align*}
  0= &2\I\int_0^T  \int_{r_2}^{\ell} \frac{1}{\sqrt{\det\Im(H(r))}}\,\int_{\R^{n-1}} e^{- |y|^2} (\eta  \:\overline{\mathcal{T} \Tilde{\eta}}+ \Tilde{\eta}\mathcal{T}\eta)(r,0)\\& \qquad\times w_0(t,r,0) \, ( a^1_{K}-a^2_K)(r,0)  \df y\, \df r\,\df t.
 \end{align*}
  Next choosing %$c_n^2= \left(\int_{\R^{n-1}} e^{- |y|^2} \df y\right)^{-1}$ and 
$w_0(t,r,0)=\phi_1(t)\, \phi_2(r)$ for some $C^{\infty}_c$ functions in their respected variables, from above we deduce that,
\begin{align*}
 0= 2\I \int_{r_2}^\ell \frac{1}{\sqrt{\det(\Im H(r))}} \, \phi_2(r)(\eta  \:\overline{\mathcal{T} \Tilde{\eta}}+ \Tilde{\eta}\mathcal{T}\eta)(r,0)\, ( a^1_{K}-a^2_K)(r,0) \df r.
\end{align*}
Further choosing $\eta=1$ in the support of $\Tilde{\eta}$ we obtain
\begin{align*}
   0=& 2\I \int_{r_2}^\ell \frac{1}{\sqrt{\det(\Im H(r))}} \, \phi_2(r)(  \PD_r\Tilde{\eta}+\frac{1}{2}\Tilde{\eta}\overline{\Delta_g\psi(r,0)}  + \frac{1}{2} \Tilde{\eta}\Delta_g\psi(r,0))\, ( a^1_{K}-a^2_K)(r,0) \df r\\
    &=2\I \int_{r_2}^\ell \frac{1}{\sqrt{\det(\Im H(r))}} \, \phi_2(r)  [\PD_r\Tilde{\eta}+ \Tilde{\eta} \,\mathrm{Re} (tr(H))(r)]\, ( a^1_{K}-a^2_K)(r,0) \df r\\
   &= 2\I \int_{r_2}^\ell \frac{\phi_2(r)}{\sqrt{\det(\Im H(r))}} \, \,  e^{-\int_{r_0}^r  \mathrm{Re} (tr(H))(s)\df s}  \PD_r \big[\Tilde{\eta} \, e^{\int_{r_0}^r  \mathrm{Re} (tr(H))(s)\df s}\big]\, ( a^1_{K}-a^2_K)(r,0) \df r.
\end{align*}
Let $r_3 \in (r_2, \ell)$.
We can choose $\Tilde{\eta} \in C_c^{\infty}$ such that $\Tilde{\eta}\, e^{\int_{r_0}^r  \mathrm{Re} (tr(H))(s)\df s}  $ converges to the indicator function of $[r_1, r_3]$ near $\gamma$. Then $ \PD_r [\Tilde{\eta}\, e^{\int_{r_0}^r  \mathrm{Re} (tr(H))(s)\df s} ]$ will converge to $ \delta_{r_1}-\delta_{r_3}$, and
\begin{align*}
  \frac{\phi_2(r_3)}{{\sqrt{\det(\Im H(r_3))}}}\, e^{-\int_{r_0}^{r_3}  \mathrm{Re} (tr(H))(s)\df s}   (a^1_{K}-a^2_K)(r_3,0)=0.
\end{align*} 
Since $ \phi_2$ solves some transport equation with non-zero initial condition on $\gamma$, this implies $\phi_2(r_3)\neq 0$, and  $ \Im H(r)$ is positive definite entails that $ \det(\Im H)(r_3)\neq 0$. This along with preceding equation implies $  e^{\int_{r_0}^{r_3}  \mathrm{Re} (tr(H))(s)\df s}  (a^1_{K}-a^2_K)(r_3,0)=0$. Moreover, using $ e^{\int_{r_0}^{r_3}  \mathrm{Re} (tr(H))(s)\df s}  \neq 0$ we conclude $ (a^1_{K}-a^2_K)(r_3,0)=0${\color{red}.} 
We have shown \eqref{induction_argument} in the case $k = K$ and $\alpha = 0$.

\smallskip
\textbf{Step 3.}  We
show that $(a^1_{K,\alpha}-a^2_{K,\alpha})(r,0)=0$  for any multi-index $ \alpha$ with $ |\alpha|=p$.

\smallskip
Here $a^j_{K}(r,y) = \sum_{\alpha} a^j_{K,\alpha}(r) y^\alpha$.
The proof is by induction in $|\alpha| = p$ and suppose that the claim holds for $p$. Consider a multi-index $\alpha$ such that $|\alpha| = p + 1$.
We argue as in \textbf{Step~2} but rescale the Gaussian beam $w$ by $y^\beta$ with $|\beta| = p + 1$, cf. Remark \ref{rem_y_shift}.
Then 
\begin{align}
     I &= 2\,\I \int_0^T  \int_{r_2}^{\ell}\int_{\R^{n-1}} e^{-2\tau \Im \psi} (\eta  \:\overline{\mathcal{T} \Tilde{\eta}}+ \Tilde{\eta}\mathcal{T}\eta) w_{0} y^\beta (a^1_{K,\alpha}-a^2_{K,\alpha})(r)\, y^{\alpha} \sqrt{|g|}\, \df r\,\df y\,\df t \\&\quad + \mathcal{O}(|y|^{2p+3}). 
 \end{align}
We now proceed exactly as above  to conclude
 \begin{align}
     I= &\frac{2\I }{\tau^{(n-1)/2} }\tau^{-(2p+2)} \int_0^T  \int_{r_2}^{\ell} \frac{1}{\sqrt{\det\Im(H(r))}}\,\int_{\R^{n-1}} e^{- |y|^2-\tau^{-1/2} \mathcal{O}(|y|^3)} (\eta  \:\overline{\mathcal{T} \Tilde{\eta}}+ \Tilde{\eta}\mathcal{T}\eta)(r,(\cdot)/\sqrt{\tau})\\& \qquad\times w_{0}(t,r)\, y^\beta  \, ( a^1_{K,\alpha}-a^2_{K,\alpha})(r)\,y^{\alpha}\,\sqrt{|g|}(r,\cdot/\sqrt{\tau})\, \df y\, \df r\,\df t + \mathcal{O}(\tau^{-(2p+3)/2}).
 \end{align}
Multiplying above by $\tau^{K+2p+1}$
and letting $\tau \rightarrow \infty$ we conclude that
    \begin{align}
   0= \frac{2\I } {\tau^{(n-1)/2} } \int_0^T  \int_{r_2}^{\ell} \frac{1}{\sqrt{\det\Im(H(r))}}\,&\int_{\R^{n-1}} e^{- |y|^2} y^\beta y^{\alpha}\, ( a^1_{K,\alpha}-a^2_{K,\alpha})(r)\, \\& \times(\eta  \:\overline{\mathcal{T} \Tilde{\eta}}+ \Tilde{\eta}\mathcal{T}\eta)(r,(\cdot)/\sqrt{\tau})\, w_{0}(t,r) \, \df y\, \df r\,\df t.
 \end{align}
 Further choosing $w_{0}(t,r,0)=\phi_1(t)\, \phi_2(r)$ for some $C^{\infty}_c$ functions in their respected variables, from above we deduce that,
\begin{align}
0=\int_{\R^{n-1}} e^{- |y|^2} y^{\beta} y^{\alpha}\,\df y\times
2\I \int_{r_2}^\ell \frac{1}{\sqrt{\det(\Im H(r))}} \,  ( a^1_{K,\alpha}-a^2_{K,\alpha})(r) \phi_2(r)(\eta  \:\overline{\mathcal{T} \Tilde{\eta}}+ \Tilde{\eta}\mathcal{T}\eta)(r,0)\,  \df r.    
\end{align}
This along with below Lemma \ref{lm:seprating tensor_n_bigger_3} implies 
\begin{align}
0= \int_{r_2}^\ell \frac{1}{\sqrt{\det(\Im H(r))}} \,  ( a^1_{K,\alpha}-a^2_{K,\alpha})(r) \phi_2(r)(\eta  \:\overline{\mathcal{T} \Tilde{\eta}}+ \Tilde{\eta}\mathcal{T}\eta)(r,0)\,  \df r.    
\end{align}
Next proceed similarly as in \textbf{Step 2} one can conclude that 
\begin{align}
  ( a^1_{K,\alpha}-a^2_{K,\alpha})(r)=0 \quad \mbox{for } \,\,\abs{\alpha}=p+1.  
\end{align}
This completes the induction step as well as the  proof of Lemma \ref{lm:determining_amplitudes}.
\end{proof}
\begin{lemma}\label{lm:seprating tensor_n_bigger_3}
     Let $n\ge 1$.  Suppose $ \int_{\R^{n}} e^{- |y|^2}  a_{\alpha} y^{\alpha}\, y^{\beta} \df y=0 $ for any multi-index $\beta$ and for some constant tensor $a_{\alpha}$, then $a_{\alpha}=0$.
\end{lemma}
\begin{remark}
 Note that we prove Lemma \ref{lm:seprating tensor_n_bigger_3} for any multi-index $ \beta$ with $\abs{\alpha}=\abs{\beta}$. This is stronger than Lemma \ref{lm:seprating tensor_n_bigger_3} and we need this in our analysis.   
\end{remark}
\begin{proof}
We start with writing $ a_{\alpha} y^{\alpha} y^{\beta}$ as 
\[ a_{\alpha} y^{\alpha} y^{\beta} = a_{\alpha_1\cdots \alpha_n} y^{\alpha_1}_1  \ldots y^{\alpha_n}_n\, y^{\beta_1}_1\,  \ldots \, y^{\beta_n}_n\,\,\mbox{with} \,\, \abs{\alpha}=\abs{\beta}=k. \]
Here we used Einstein summation convention for repeated indices. Note that the integral $ \int_{\R^{n}} e^{- |y|^2}  a_{\alpha} y^{\alpha}\, y^{\beta} \df y$ is non-zero when $ \alpha_i+ \beta_i$ is even for all $ 1\le i\le n$. Moreover, one has
\begin{equation}
    \begin{aligned}\label{eq_4.3}
  0&= \int_{\R^{n}} e^{- |y|^2}  a_{\alpha} y^{\alpha}\, y^{\beta} \df y \\&=a_{\alpha_1\cdots \alpha_n} \,  \prod_{i=1}^n \Gamma(\frac{\alpha_i+ \beta_i+1}{2})  \quad 
  \mbox{when $ \alpha_i+ \beta_i$ is even for all $ 1\le i\le n$}.
\end{aligned}
\end{equation}
Thus proving Lemma \ref{lm:seprating tensor_n_bigger_3} is equivalent to show that \eqref{eq_4.3} implies $ a_{\alpha}=0$ for any multi-index $ \alpha$ with $ \abs{\alpha}=\abs{\beta}$.

 This holds trivially when $n=1$, because $a_{\alpha}$ is a scalar function. We next argue by induction and assume that
\eqref{eq_4.3} implies $ a_{\alpha}=0$ in any dimension $n=N$. Then for $n=N+1$  we denote $ \alpha= (\alpha_1,\alpha_2,\ldots, \alpha_{N+1})=(\alpha', \alpha_{N+1})$ and $ \beta=(\beta_1,\beta_2, \ldots, \beta_{N+1}) = (\beta', \beta_{N+1})$, where $ \abs{\alpha}=k= \abs{\beta}$. For any $ \alpha_{N+1}$ we define
\begin{align}
   f_{\alpha_{N+1}} := a_{\alpha_1\alpha_2\ldots 
   \alpha_N \alpha_{N+1}}   \prod_{i=1}^n \Gamma(\frac{\alpha_i+ \beta_i+1}{2}).%a_{\alpha',\alpha_{N+1}) \Gamma(\frac{\alpha'+\beta'+1}{}. 
   % \,y^{\alpha_1}_1 y^{\alpha_2}_2 \, y^{\beta_1}_1\,y^{\beta_2}_2,\quad \mbox{where} \quad a_{\alpha_1\alpha_2\alpha_3\ldots \alpha_n}  y^{\alpha_3}_3  \ldots y^{\alpha_n}_n\, y^{\beta_3}_3\,  \ldots \, y^{\beta_n}_n.
\end{align}
By assumption we have that  \eqref{eq_4.3}
 is true for $n=N+1$. Hence,
 \begin{align}
\Gamma(\frac{\alpha_{N+1}+\beta_{N+1}+1}{2}) f_{\alpha_{N+1}}=0  \quad \mbox{for all $ 0\le \beta_{N+1} \le k$ and  $\alpha_{N+1}+\beta_{N+1}$ is even}. 
\end{align}
Next we assume that $ k=2p$ and $ \alpha_{N+1}+\beta_{N+1}= 2l$ where $ 0\le l\le 2k.$ Then the preceding system of equations can be written as 
\begin{align}
   \sum_{2l-\beta_{N+1}=0}^{2p} \Gamma(l+\frac{1}{2}) f_{2l-\beta_{N+1}}=0.
\end{align}
Further denoting $ \beta_{N+1} =2s$ from above we obtain 
\begin{align}
     \sum_{l=s}^{p+s} \Gamma(l+\frac{1}{2}) f_{2l-2s}=0\quad 0\le s\le p.
\end{align}
This can be written as a matrix equation $A_{p+1}X=0$, where $A_{p+1}$  is  given by
\begin{align}
    A_{p+1}= \begin{pmatrix}
        \Gamma(\frac{1}{2}) & \Gamma(\frac{3}{2}) & \cdots &\Gamma(p-\frac{1}{2})& \Gamma(p+\frac{1}{2})\\
        \Gamma(\frac{3}{2}) & \Gamma(\frac{5}{2}) & \cdots &\Gamma(p+\frac{1}{2})&\Gamma(p+\frac{3}{2}) \\
        \vdots&\vdots& \ddots& \vdots&\vdots\\
      \Gamma(p+\frac{1}{2})& \Gamma(p+\frac{3}{2})&\cdots& \Gamma(2p-\frac{1}{2})&  \Gamma(2p+\frac{1}{2}) 
    \end{pmatrix}.
\end{align}
We next show that $\det (A_{p+1}) \neq 0$. To achieve this we perform the following column operation: $(j+1)$-th column - $j$-th column  $ \times (j-\frac{1}{2})$ for $ 1\le j\le p$. This implies
\begin{align*}
    \det (A_{p+1}) = p!\, \Gamma(\frac{1}{2})\, \det (A_p),
\end{align*}
 where $A_p$ is given by 
 \begin{align*}
     A_p = \begin{pmatrix}
         %  \Gamma(\frac{1}{2}) &\Gamma(\frac{3}{2}) & \cdots &\Gamma(p-\frac{1}{2})& \Gamma(p+\frac{1}{2})\\
        \Gamma(\frac{3}{2}) & \Gamma(\frac{5}{2}) & \cdots &\Gamma(p+\frac{1}{2}) \\
        \Gamma(\frac{5}{2}) & \Gamma(\frac{7}{2})& \cdots & \Gamma(p+\frac{3}{2})\\
        \vdots&\vdots& \ddots& \vdots\\
      \Gamma(p+\frac{1}{2})& \Gamma(p+\frac{3}{2})&\cdots& \Gamma(2p-\frac{1}{2}) 
     \end{pmatrix}.
 \end{align*}
 Next we do the following column operation: $(j+1)$-th column - $j$-th column  $ \times (j+\frac{1}{2})$ for $ 1\le j\le p-1$. This further entails
 \begin{align*}
     \det(A_{p})=  (p-1)!\, \Gamma(\frac{3}{2})\, \det (A_{p-1})\implies \det(A_{p+1})= p!\, (p-1)!\, \Gamma(\frac{1}{2})\, \Gamma(\frac{3}{2})\, \det(A_{p-1}).
 \end{align*}
Continuing in this way after finitely many steps we obtain 
\begin{align*}
 \det(A_{p+1})= \prod_{j=1}^{p-1}  (p+1-j) !\,\Gamma(j-\frac{1}{2})\,\det(A_2),
\end{align*}
where $ A_2= \begin{pmatrix}
 \Gamma(p-\frac{1}{2})& \Gamma(p+\frac{1}{2})\\
  \Gamma(p+\frac{1}{2})& \Gamma(p+\frac{3}{2})
\end{pmatrix}$. Since $\det(A_2)= \Gamma(p-\frac{1}{2})\times \Gamma(p+\frac{1}{2}) $, this further entails 
\begin{align*}
    \det(A_{p+1})=  \prod_{j=1}^{p+1}  (p+1-j) !\,\Gamma(j-\frac{1}{2}).
\end{align*}
This clearly shows that $\det(A_{p+1})\neq 0$ and   $f_{\alpha_{N+1}} =0$ for any fixed $ 0\le \alpha_{N+1}\le k$.
%Clearly $\det A \neq 0$. This implies $ f_{2l}=0$ for $0\le l\le p.$ Similarly one can show that $F_{2l-1}=0$ for $ 1\le l\le p.$ This implies 
%\begin{align}
  %  f_{\alpha_{N+1}} =0 \quad \mbox{for any fixed $ 0\le \alpha_{N+1}\le k$}.
%\end{align}
Furthermore, choosing $\alpha_{N+1}=\beta_{N+1}$ we deduce that
\begin{align}
  0=a_{\alpha_1\alpha_2\ldots 
   \alpha_N \alpha_{N+1}}   \prod_{i=1}^N \Gamma(\frac{\alpha_i+ \beta_i+1}{2}),  
\end{align}
for all multi-index $\alpha' $ and $\beta'$ in dimension $N$ with $ \abs{\alpha'}=\abs{\beta'}$. Therefore by induction hypotheses, this implies 
\begin{align}
     a_{\alpha_1\alpha_2\ldots 
   \alpha_N \alpha_{N+1}}=0, \quad\mbox{for any multi-index $ \alpha$ with $ \abs{\alpha}=k$.}
\end{align}
This completes the proof.
% This along with \eqref{eq_4.3} entails
% \begin{align}
%     0= f_{\alpha_1\alpha_2} \Gamma(\frac{\alpha_1+ \beta_2+1}{2})\Gamma(\frac{\alpha_1+ \beta_2+1}{2}) \quad \mbox{}
% \end{align}
% for all multi-index $(\alpha_1,\alpha_2)$ and  $ (\beta_1,\beta_2)$ with $\alpha_1+\alpha_2=k-\abs{\alpha'}, \beta_1+\beta_2=\alpha -\abs{\beta'}$. In order to apply result from Lemma \ref{lm:seprating tensor_n=2} we choose $\abs{\alpha'}= \abs{\beta'}. $ Then applying the result from Lemma \ref{lm:seprating tensor_n=2} we obtain that
% \[   f_{\alpha_1\alpha_2}=0 \]
\end{proof}

\subsection{Proof of Theorem \ref{th:main_result}}\label{subsec:main_result_thm_1.2}
We begin by recalling that $M$ is an admissible manifold in the sense of Definition \ref{def_admisssible}. Therefore, we fix $p\in M\setminus\Omega$ and consider $\lambda>0$, $\lambda':=\sqrt{1-\lambda^2}$, and $\xi_0,\xi_1,\xi_2$ as in Definition \ref{def_admisssible} with $\xi_0:=\lambda'\xi_1+\lambda\xi_2$. For $j=0,1,2$ we consider Fermi coordinates along $\gamma_{p,\xi_j}$ such that $p=(0,0\cdots,0)$ in these coordinates. We construct an approximate Gaussian beam $U_\tau^{(j)}$ in these coordinates along $\gamma_{p,\xi_j}$, and denote the corresponding exact solution of the Schr\"odinger equation by $u_\tau^{(j)}=U_\tau^{(j)}+R_\tau^{(j)}$, where the remainder $R_\tau^{(j)}$ is constructed as in \eqref{eq_rmdr}.

We again let $\eta\in C^\infty_0(M)$ be such that $\eta=1$ in $M\setminus\Omega$, and consider the sources $f^{(j)}_\tau=(i\p_t+\Delta+V)(\eta u_\tau^{(j)})$. Then the function $\eta u^{(j)}_\tau$ solves

\begin{align}
\begin{cases}
i \p_t \mathcal{U}^{(j)}_\tau + \Delta \mathcal{U}^{(j)}_\tau + V \mathcal{U}^{(j)}_\tau = f^{(j)}_\tau & \text{on $(0,T) \times M$,}
\\
\mathcal{U}^{(j)}_\tau|_{x \in \p M} = 0,
\\
\mathcal{U}^{(j)}_\tau|_{t = 0} = 0.
\end{cases}
    \end{align}
As before, we choose $\kappa$ large enough that $H^{2\kappa}_0$ is a Banach algebra. It can then be shown that the source $f^{(j)}_\tau\in H^{2\kappa}_0$ is determined by the source-to-solution map $L_{\beta,V}$, up to any error $O(\tau^{-K})$ in the $H^{2\kappa}$-norm. The proof of this fact follows, as in the case of geometric optics, from the fact that $[(\I\p_t+\Delta+V),\eta]=2\pair{d\eta,d}_g+\Delta\eta, $ whence $f^{(j)}_\tau$ is given by the expression $f^{(j)}_\tau=2\pair{d\eta,dU^{(j)}_\tau}_g+\Delta\eta U^{(j)}_\tau+2\pair{d\eta,dR^{(j)}_\tau}_g+\Delta\eta R^{(j)}_\tau. $ We note that the first two terms on the right-hand side are uniquely determined by $L_{\beta,V}$ using Lemma \ref{lm:determining_amplitudes}. 
%\TODO{It remains to prove that $L_{\beta,V}$ determines the amplitudes of the approximate Gaussian beam solution $U_\tau$. In particular, we need to deal with the fact that $M$ may not be convex, which was used in the analogous proof in section \ref{e_s3.2}. Let's also consider moving this result and that of section \ref{e_s3.2} to the appendices.} 
On the other hand, we can apply estimate \eqref{!GBRest} to the last two terms to conclude that they are $O(\tau^{\frac{3-N}{2}+4\kappa+2})$ in the $H^{2\kappa}$-norm.

We next implement the second order linearization technique. To this end, let $\varepsilon_1,\varepsilon_2>0$ be small, we fix $\varepsilon=(\varepsilon_1,\varepsilon_2)$ and define the source term $f_\tau=\varepsilon_1f^{(1)}_{\lambda'\tau}+\varepsilon_2f^{(2)}_{\lambda\tau}$. For small enough $\varepsilon$, it holds that $f_\tau\in\mathcal{H}$, and we observe that\[-\frac{1}{2}\p_{\varepsilon_1}\p_{\varepsilon_2}L_{\beta,V}f_\tau|_{\varepsilon=0}=w_\tau|_{(0,T)\times\Omega},\]with $w_\tau$ the solution of the linear Schr\"odinger equation\begin{align}\begin{cases}i\p_tw_\tau+\Delta w_\tau+Vw_\tau=\beta\,\mathcal{U}^{(1)}_{\lambda'\tau}\,\mathcal{U}^{(2)}_{\lambda\tau}\\w_\tau|_{x \in \p M} = 0\\w_\tau|_{t=0} = 0.
\end{cases}\end{align}
Then it follows that\begin{equation}\label{!before}-\frac{1}{2}\int_{(0,T)\times\Omega}\p_{\varepsilon_1}\p_{\varepsilon_2}L_{\beta,V}f_\tau|_{\varepsilon=0}\overline{f^{(0)}_\tau}\df V_g\, \df t=\int_{(0,T)\times M}w_\tau\overline{(i\p_t+\Delta+V)\mathcal{U}^{(0)}_\tau}\df V_g\df t.\end{equation}Recall that $L_{\beta,V}$ is a continuous map from $\mathcal{H}$ in to $H^{2\kappa}_0$, and that $f_\tau\in\mathcal{H}$ is determined by $L_{\beta,V}$ up to $O(\tau^{-\frac{3-N}{2}+4\kappa+2})$. Thus $L_{\beta,V}$ determines the left-hand side of \eqref{!before} up to this error. Integrating the right-hand side by parts, we observe that it is given by\begin{equation}\label{!after}\int_{(0,T)\times M}\beta\overline{\mathcal{U}^{(0)}_\tau}\mathcal{U}^{(1)}_{\lambda'\tau}\mathcal{U}^{(2)}_{\lambda\tau}\, \df V_g\df t.
\end{equation}
We would like to approximate $\mathcal{U}^{(j)}_\tau$ by $\eta U^{(j)}_\tau$ in \eqref{!after}. Using the bound \eqref{!GBRest} for $\eta R^{(j)}_\tau$, together with the estimate \eqref{!estimates} for $U^{(j)}_\tau$ and the fact that $H^{2\kappa}_0$ is a Banach algebra, we conclude that $L_{\beta,V}$ determines the integral
\[\int_{(0,T)\times M}\beta\eta^3 \overline{U^{(0)}_\tau} U^{(1)}_{\lambda'\tau} U^{(2)}_{\lambda\tau} \df V_g\df t,\]
up to an error of $O(\tau^{-K})$ order, provided that we have chosen the order $N$ of the Gaussian beams to be sufficiently large. Recall that $\eta=1$ in $\supp\big(\overline{U^{(0)}_\tau}U^{(1)}_\tau U^{(2)}_\tau\big)$ and that $\phi(t)$ in the definition of $a^{(j)}_\tau$ is arbitrary. Thus, the source-to-solution map determines, for all $t\in(0,T)$, the integral\[\mathcal{I}=\int_M\beta(t,\cdot)\overline{U^{(0)}_\tau}(t,\cdot)U^{(1)}_{\lambda'\tau}(t,\cdot)U^{(2)}_{\lambda\tau}(t,\cdot)\df V_g,\]up to any polynomial error $O(\tau^{-K})$. Henceforth, we suppress this $t$-dependence in our notation. In order to recover the pair $(\beta,V)$, we shall use the method of stationary phase to analyze the integral $\mathcal{I}$. Therefore, let us begin with the following result.

\begin{lemma}
 The function\[\Psi=-\overline{\psi^{(0)}}+\lambda'\psi^{(1)}+\lambda\psi^{(2)}\]satisfies the following two conditions:
\begin{enumerate}[(i)]
\item$\nabla_g\Psi(p)=0$.
\item $\det(\nabla_g^2 \Psi)(p)\neq 0$.
\end{enumerate}
Where $\psi^{(j)}$ are approximate Gaussian beam solutions along the geodesics $\gamma_{p,\xi_j}$  for $j=0,1,2$, and recall that these geodesics intersect only at $p\in M$.
\end{lemma}

\begin{proof}
First of all, we note that $\nabla_g\psi^{(j)}(p)=\xi_j$ for $j=0,1,2$. Thus, we observe that\[\nabla_g\Psi(p)=-\xi_0+\lambda'\xi_1+\lambda\xi_2,\]which vanishes by the definitions of $\lambda$, $\lambda'$ and $\xi_0$. For the last claim, it suffices to show that $D^2\Im\Psi(X,X)>0$ for $X\in T_pM\setminus 0$. We first note that\[\Im\Psi=\Im\psi^{(0)}+\lambda'\Im\psi^{(1)}+\lambda\Im\psi^{(2)},\]
which implies that $D^2\Im\Psi(X,X)\geq0$. Indeed, in the system of Fermi coordinates along $\gamma_{p,\xi_j}$, it holds that
\[D^2\Im\psi^{(j)}|_{\gamma_{p,\xi_j}}=\begin{pmatrix}0&0\\0&H_j\end{pmatrix},\] 
where $H_j$ is the appropriate matrix solving \eqref{!riccati} on $\gamma_{p,\xi_j}$. Therefore, using \eqref{!psi01} and $\Im(\psi)(r,y)\geq c|y|^2$  we observe that for $j=0,1,2$ we have
\[\begin{split}D^2\Im\psi^{(j)}(X,X)&\geq0\quad\forall X\in T_pM,\\D^2\Im\psi^{(j)}(X,X)&>0\quad\forall X\in T_pM\setminus\vspan(\xi^{(j)}).\end{split}\]
Since $\xi_1$ and $\xi_2$ are linearly independent, the claimed result follows.
\end{proof}

We now turn to analyzing the integral $\mathcal{I}$. Since $\gamma_{p,\xi_0}\cap\gamma_{p,\xi_1}\cap\gamma_{p,\xi_2}=\{p\}$, we observe that the product\[\overline{U^{(0)}_\tau}U^{(1)}_{\lambda'\tau}U^{(2)}_{\lambda\tau}=e^{i\tau\Psi}a^{(0)}_\tau a^{(1)}_{\lambda'\tau}a^{(2)}_{\lambda\tau}\]is supported in a neighbourhood of $p$. We can expand the amplitudes $a^{(j)}_\tau$ in terms of functions $v^{(j)}_k$ as in \eqref{eq transport}, and apply the method of stationary phase (for example, see \cite{Hormander_1}*{Theorem 7.7.5}) to the integral $\mathcal{I}$ termwise after expanding. Thus we recover\[\begin{split}&\tau^{\frac{n}{2}}e^{i\tau\Psi(p)}\mathcal{I}=C_0\beta(p)v^{(0)}_0v^{(1)}_0v^{(2)}_0(p)\\+&\tau^{-1}C_1\beta(p)\Big(v^{(0)}_1v^{(1)}_0v^{(2)}_0(p)+\frac{1}{\lambda'}v^{(0)}_0v^{(1)}_1v^{(2)}_0(p)+\frac{1}{\lambda} v^{(0)}_0v^{(1)}_0v^{(2)}_1(p)\Big)+O(\tau^{-2}),\end{split}\]
where $C_0,C_1$ are known constants which do not depend on the pair $(\beta,V)$. In particular, by recalling that $v^{(j)}_0(p)=1$, we are able to recover $\beta(p)$ from the first term of the above. Since $p\in M\setminus\Omega$ was arbitrary, this is sufficient to recover $\beta$ in its entirety. On the other hand, since $\beta$ is non-zero in $\supp(V)$, we recover from the second term the quantity\[v^{(0)}_1(p)+\frac{1}{\lambda'}v^{(1)}_1(p)+\frac{1}{\lambda} v^{(2)}_1(p).\] 
Then, since $\lambda'=\sqrt{1-\lambda^2}$ we can, in fact, obtain the quantity\[\lim_{\lambda\rightarrow0}\ \lambda\Big(v^{(0)}_1(p)+\frac{1}{\lambda'}v^{(1)}_1(p)+\frac{1}{\lambda} v^{(2)}_1(p)\Big)=v^{(2)}_1(p).\] Finally, using the splitting \eqref{!split} together with the fact that $b_{1,0}$ does not depend on the potential $V$, we recover $c_{1,0}$ evaluated at $p$, which, letting $p=\gamma(r_1)$, we can compute to be\begin{equation}\label{!last1}c_{1,0}(0)=\frac{i}{2}\int_{r_0}^{r_1}V(s,0)ds\end{equation}in the Fermi coordinates along $\gamma_{p,\xi_2}$. We can recover $V(p)$ by differentiating \eqref{!last1} with respect to the upper limit of integration, and this completes the proof of Theorem \ref{th:main_result}.
% This along with \eqref{eq_4.3} entails
% \begin{align}
%     0= f_{\alpha_1\alpha_2} \Gamma(\frac{\alpha_1+ \beta_2+1}{2})\Gamma(\frac{\alpha_1+ \beta_2+1}{2}) \quad \mbox{}
% \end{align}
% for all multi-index $(\alpha_1,\alpha_2)$ and  $ (\beta_1,\beta_2)$ with $\alpha_1+\alpha_2=k-\abs{\alpha'}, \beta_1+\beta_2=\alpha -\abs{\beta'}$. In order to apply result from Lemma \ref{lm:seprating tensor_n=2} we choose $\abs{\alpha'}= \abs{\beta'}. $ Then applying the result from Lemma \ref{lm:seprating tensor_n=2} we obtain that
% \[   f_{\alpha_1\alpha_2}=0 \]

%\end{proof}
\section*{acknowledgements}
The authors are thankful for Ali Feizmohammadi on discussions that were crucial for the ideas used in Section $3.2$. M.L., L.O., M.S were
supported by the Finnish Centre of Excellence of Inverse Modelling and Imaging. AT was supported by EPSRC DTP studentship EP/N509577/1. M.L was also partially supported by Academy of Finland, grants 284715, 312110 and M.S. was supported by ERC.

%\bibliography{ref}
 %\bibliographystyle{siam}
 \end{document}